\documentclass[12pt]{amsart}

\usepackage{amsmath}
\usepackage{amssymb}
\usepackage{amsfonts}
\usepackage{amsthm}
\usepackage{enumerate}
\usepackage{color}
\usepackage{psfrag}
\usepackage{mathrsfs}
\usepackage{amscd}
\usepackage[hyperindex,backref=page]{hyperref}
\usepackage{graphicx,color}
\usepackage{tikz} 

\newtheorem{thm}{Theorem}[section]
\newtheorem{cor}[thm]{Corollary}
\newtheorem{lemma}[thm]{Lemma}
\newtheorem{prop}[thm]{Proposition}

\newtheorem{q}[thm]{Question}
\newtheorem{conj}[thm]{Conjecture}

\theoremstyle{definition}
\newtheorem{defi}[thm]{Definition}

\theoremstyle{remark}

\DeclareMathOperator{\link}{link}
\DeclareMathOperator{\diam}{diam}

\DeclareMathOperator{\depth}{depth}
\DeclareMathOperator{\pd}{pd}
\DeclareMathOperator{\height}{ht}
\DeclareMathOperator{\Ass}{Ass}
\DeclareMathOperator{\supp}{supp}
\DeclareMathOperator{\reg}{reg}
\DeclareMathOperator{\im}{im}
\DeclareMathOperator{\Min}{Min}

\DeclareMathOperator{\lcm}{lcm}
\DeclareMathOperator{\Ann}{Ann}
\DeclareMathOperator{\ogirth}{odd-girth}

\DeclareMathOperator{\assrad}{assrad}

\begin{document}


\title[Comparisons between ordinary and symbolic powers of edge ideals]{Comparisons between ordinary and symbolic powers of edge ideals with respect to regularity and depth}

\author{Yuji Muta}

\address[Y. Muta]{Department of Mathematics, Okayama University, 3-1-1 Tsushima-naka, Kita-ku, Okayama 700-8530 Japan.}

\email{p8w80ole@s.okayama-u.ac.jp}

\begin{abstract}
In this paper, we investigate the difference between ordinary and symbolic powers of edge ideals on the Castelnuovo--Mumford regularity and the depth. As a main theorem, we prove that the regularities of ordinary and symbolic powers coincide for edge ideals of simplicial graphs, as a partial result of Minh's conjecture. For the depth, we first show that, for each $k\geq2$, the inequality $\depth S/I^{(k)}\geq\depth S/I^{k}$ does not hold for squarefree monomial ideals in general. On the other hand, we prove that it holds for edge ideals when $k=2,3$. We also study the symbolic-ordinary discrepancy module $I(G)^{(k)}/I(G)^{k}$ of the edge ideal $I(G)$ of a graph $G$ and give a graph-theoretic formula for its Krull dimension in terms of induced odd cycles of $G$, thereby answering a question and a problem posed by Ha and Minh. 
\end{abstract}


\subjclass[2020]{Primary: 13F55, 13H10, 05C75; Secondary: 13D45, 05C90, 55U10}


\keywords{Minh's conjecture, symbolic power, ordinary power, edge ideal, depth, regularity, local cohomology}


\thanks{}
	

\maketitle


\section{Introduction}

The study of ordinary and symbolic powers of ideals is a classical topic in commutative algebra. Although these two kinds of powers are closely related, their algebraic properties can be quite different. Thus, it is natural to compare homological invariants of ordinary and symbolic powers. In this paper, we focus on two basic invariants, the Castelnuovo--Mumford regularity (or regularity for short) and the depth, for edge ideals.

Throughout, we always assume that all graphs are finite graphs without loops and multiple edges. Let $G$ be a finite simple graph on the vertex set $[n]=\{1,\ldots, n\}$ with $E(G)\neq\emptyset$. Let $S=K[x_{1},\ldots, x_{n}]$ be the polynomial ring  in $n$ variables over a field $K$. The {\it edge ideal} of $G$ is the squarefree monomial ideal of $S$, denoted by $I(G)$, generated by quadratic squarefree monomials corresponding to the edges of $G$, that is, 
$$I(G)=(x_{i}x_{j}\,\,: \{i,j\}\in E(G)).$$
Since $I(G)$ is a squarefree monomial ideal, its $k$-th symbolic power is given by 
$$I(G)^{(k)}=\bigcap_{P\in\Min(I(G))}P^{k},$$
where $\Min(I(G))$ denotes the set of minimal primes of $I(G)$. Notice that we always have $I(G)^{k}\subset I(G)^{(k)}$ for all $k\geq1$. 

In Section \ref{reg}, we discuss the regularity. Minh proposed the following conjecture in \cite{mnptv2022}. 

\begin{conj}[\cite{mnptv2022}, Conjecture A]
Let $G$ be a graph. Then $$\reg I(G)^{(k)}=\reg I(G)^{k}\mbox{ for all }k\geq1.$$
\end{conj}

Minh's conjecture has been verified for several classes of graphs. It is known that $I(G)^{(k)}=I(G)^{k}$ for all $k\geq1$ if and only if $G$ is bipartite \cite[Theorem 5.9]{svv1994}. Hence, Minh's conjecture is true if $G$ is bipartite. The first non-bipartite case was proved by Gu, H\'{a}, O'Rourke, and Skelton for odd cycles \cite[Theorem 5.3]{ghrs2020}. After that, this conjecture was proved for several classes of graphs, including certain clique sums of odd cycles and bipartite graphs \cite{jk2020}, unicyclic graphs \cite[Theorem 3.9]{sf2020unicyclic}, Cameron--Walker graphs \cite{sf2020cw}, chordal graphs \cite{sf2022chordal}, and co--chordal graphs \cite{an2026}. Moreover, Minh, Nam, Phong, Thuy, and Vu proved that this conjecture holds for every graph when $k=2,3$ \cite{mnptv2022}. Nevertheless, this conjecture remains open for arbitrary graphs and all $k\geq1$. 

As the first main theorem of this paper, we prove that Minh's conjecture holds for simplicial graphs. Recall that a graph $G$ is called {\it simplicial} if every vertex of $G$ belongs to the closed neighborhood of some simplicial vertex, where a vertex is {\it simplicial} if its closed neighborhood is a clique. In fact, we obtain an explicit formula for the regularity of both ordinary and symbolic powers in terms of the induced matching number $\im(G)$ of $G$. 
\begin{thm}[see, Theorem \ref{main1}]
Let $G$ be a simplicial graph with $E(G)\neq\emptyset$. Then we have 
$$\reg I(G)^{(k)}=\reg I(G)^{k}=2k+\im(G)-1\mbox{ for all }k\geq1.$$
\end{thm}
Simplicial graphs include whisker graphs \cite{v1990}, multi-whisker graphs \cite{mpt2025}, clique-whiskered graphs \cite{cn2012}, multi-clique-whiskered graphs \cite{mt2025}, clique-corona graphs \cite{hp2022}, multi-clique-corona graphs \cite{mt2026}, Cohen--Macaulay chordal graphs \cite{hhz2006}, and Cohen--Macaulay Cameron--Walker graphs \cite{hhko2015}.

We will explain the outline of the proof. We first prove that the regularity of ordinary and symbolic powers of edge ideals of a whisker graph $W(G)$ of an arbitrary graph $G$ coincides with $2k+\im(W(G))-1$ for all $k\geq1$ (see, Theorem \ref{whisker}). Using a technique which was introduced by Bayati and Herzog \cite{bh2014}, we then extend this formula to multi-whisker graphs (see, Corollary \ref{multi-whisker}). The key step is a comparison of the regularity of symbolic powers of two squarefree monomial ideals. More precisely, we show that a suitable map between the minimal primes of two squarefree monomial ideals gives an inequality between the regularities of their all symbolic powers. The proof is based on polarization, Alexander duality, and lcm-lattices (see, Lemma \ref{key}). Finally, we relate every simplicial graph to a suitable multi-whisker graph and apply the above comparison result to obtain the desired formula. 

We end Section \ref{reg} with another comparison result which holds for an arbitrary graph. Let $\ogirth(G)$ denote the length of a shortest induced odd cycle of $G$. Rinaldo, Terai, and Yoshida \cite{rty2011} showed that $I(G)^{(k)}=I(G)^{k}$ if and only if $\ogirth(G)>2k-1$. Hence, if $\ogirth(G)=2t-1$, then $t$ is the first power for which the ordinary and symbolic powers differ. We show that, even at this first nontrivial power, their local cohomology modules are closely related as follows. 
\begin{thm}[see, Theorem \ref{local coho} and Corollary \ref{ineq of reg}]
Let $G$ be a graph with $\ogirth(G)=2t-1$. If $1\leq k\leq t$, then, for every $i$, the natural map $H_{\mathfrak{m}}^{i}(S/I(G)^{k})\rightarrow H_{\mathfrak{m}}^{i}(S/I(G)^{(k)})$ induced from a short exact sequence $0\rightarrow I(G)^{(k)}/I(G)^{k}\rightarrow S/I(G)^{k}\rightarrow S/I(G)^{(k)}\rightarrow0$ is surjective. 

In particular, we have $$\reg I(G)^{(k)}\leq\reg I(G)^{k}\mbox{ for all }1\leq k\leq t.$$
\end{thm}

In Section \ref{depth}, we turn to the comparison of the depth. It is natural to ask whether the depth of a symbolic power of a squarefree monomial ideal is always bounded below by that of the corresponding ordinary power.   
\begin{q}\label{ques of depth}
For a squarefree monomial ideal $I$, is it true that 
$$\depth S/I^{(k)}\geq\depth S/I^{k}\mbox{ for all }k\geq1 ?$$
\end{q}

Notice that equality cannot be expected in general. Indeed, the depth of ordinary powers may be zero, while every symbolic power of a squarefree monomial ideal without maximal ideal has positive depth. We first provide that Question \ref{ques of depth} has a counterexample in general. In fact, the inequality may fail even for squarefree monomial ideals of height two. 

\begin{thm}
For each $k\geq2$, there exists a squarefree monomial ideal $I$ with $\height I=2$ such that 
$$\depth S/I^{(k)}<\depth S/I^{k}.$$
\end{thm}

These ideals are obtained from cover ideals of whisker graphs of complete graphs.  This naturally leads us to ask for classes of squarefree monomial ideals for which the inequality in Question \ref{ques of depth} hold. In this paper, we will focus on edge ideals. Although Question \ref{ques of depth} has a negative answer in general, we do not know any counterexample among edge ideals. This leads us to propose the following conjecture. 

\begin{conj}\label{conj}
Let $G$ be a graph. Then we have 
$$\depth S/I(G)^{(k)}\geq\depth S/I(G)^{k}\mbox{ for all }k\geq1.$$
\end{conj}

As a first result toward this conjecture, we prove it for the second and third powers. 
\begin{thm}[see, Corollary \ref{small power}]
For any graph $G$, we have 
$$\depth S/I(G)^{(k)}\geq\depth S/I(G)^{k}\mbox{ for }k=2,3.$$   
\end{thm}

Our proof is based on associated radical ideals of monomial ideals. More precisely, we indeed prove $\assrad(I(G)^{(k)})\subset\assrad(I(G)^{k})$ for $k=2,3$ in Theorem \ref{assrad}, which is stronger than the depth comparison. Here, for a monomial ideal $I$, $\assrad(I)$ is defined by 
$$\assrad(I)=\{\sqrt{I:u}\,\,: u\mbox{ is a monomial with }u\notin I\}.$$

Also, for the second and fourth powers, we give a sufficient condition for the depth comparison to hold for all powers. Specifically, if $G$ is connected and $\depth S/I(G)^{2}\leq1$ or $\depth S/I(G)^{4}=0$, then $\depth S/I(G)^{(k)}\geq\depth S/I(G)^{k}$ for all $k\geq1$ (see, Proposition \ref{sufficient}). 

To obtain further results on the depth comparison, we investigate the symbolic-ordinary discrepancy module $I(G)^{(k)}/I(G)^{k}$ of the edge ideal $I(G)$. This module was studied in \cite{hm2026} using relative simplicial complexes. We give a graph-theoretic description of its support in terms of induced odd cycles (see, Theorem \ref{locus}). As a consequence, we obtain the explicit formula for $\dim I(G)^{(k)}/I(G)^{k}$. In particular, this recovers Corollary 7.8 and Theorems 7.10 and 7.11 as immediate consequence, and yields complete answers to Question 7.9 and Problem 7.13 in \cite{hm2026}. 

\begin{thm}[see, Corollary \ref{char of dim}]
Let $G$ be a graph and $k\geq1$. Then we have 
$$\dim I(G)^{(k)}/I(G)^{k}=\max\{\alpha(G-N_{G}[C])\,\,: C\in\mathcal{C}_{k}(G)\},$$
where $\mathcal{C}_{k}(G)=\{C\,\,: C\mbox{ is an induced odd cycle of }G\mbox{ with }|C|\leq 2k-1\}$ and we use the convention that $\max\emptyset=\dim 0=-\infty$.
\end{thm}

We then apply this formula to the depth comparison. As a result, if $G$ has no induced subgraph isomorphic to  $(C_{2r+1}\sqcup2K_{1})$ for all $1\leq r\leq k-1$, then Conjecture \ref{conj} holds for this $k$ (see, Corollary \ref{depth comparison free}).  


\section{Preliminaries}

In this section, we recall several definitions and known results from combinatorial commutative algebra that will be used this paper. We refer the reader to \cite{BHBook, HHBook, VBook} for the detailed information. Readers who are already familiar with these concepts may skip to the next section. 

\subsection{Graphs and simplicial complexes}
Let $G$ be a finite graph without loops and multiple edge. We denote by $V(G)$ the vertex set of $G$ and $E(G)$ the edge set of $G$. Set $V(G)=[n]=\{1,\ldots, n\}$ and let $S=K[x_{1},\dots,x_{n}]$ be the polynomial ring in $n$ variables over a field $K$. Then the {\it edge ideal} of $G$, denoted by $I(G)$, defined as 
$$I(G)=(x_{i}x_{j}\,\,:\{i,j\}\in E(G)).$$
Edge ideals were introduced by Villarreal in \cite{v1990}. For a subset $W\subset V(G)$, we denote $G|_{W}$ the induced subgraph $W$. A subset $C$ of $V(G)$ is called {\it vertex cover}, if $C\cap e\neq\emptyset$ for every edge $e$ of $G$. Also, if $C$ is a vertex cover that is minimal with respect to inclusion, then $C$ is called a {\it minimal vertex cover} of $G$. The \textit{cover ideal}, denoted by $J(G)$, of $G$ is the monomial ideal $$J(G)=(\prod_{i\in W}x_{i}\,\,: W \mbox{ is a minimal vertex cover of } G).$$ 
Moreover, a subset $A$ of $V(G)$ is called an {\it independent set}, if no two vertices in $A$ are adjacent to each other. In particular, if $A$ is an independent set of $G$ that is maximal with respect to inclusion, then $A$ is called a {\it maximal independent set} of $G$. We set $$\alpha(G)=\max\{|A|\,\,: A\mbox{ is an independent set of }G\}.$$
The {\it independence complex} of $G$ is the set of independent sets of $G$, which forms a simplicial complex $\Delta(G)$. A subset $M$ of $E(G)$ is called a {\it matching}, if no two edges in $M$ share a common vertex. Furthermore, if there is no edge in $E(G)\setminus M$ that is contained in the union of edges of $M$, then $M$ is called an {\it induced matching}. Then, the {\it induced matching number} of $G$, denoted by ${\rm im}(G)$, is defined as $${\rm im}(G)=\max\{|M|\,\,: M\mbox{ is an induced matching of }G\}.$$
For a subset $A\subset V(G)$, we denote by $N(A)$ and $N[A]$ the open and closed neighborhoods of $A$, respectively, and write $G-A=G|_{V(G)\setminus A}$. Also, we denote by $G^{c}$ the complement graph of $G$. 

A \emph{simplicial complex} $\Delta$ on $[n]=\{1,\ldots,n\}$ is a collection of subsets of $[n]$ closed under taking subsets, that is, $F\in\Delta$ and $H\subseteq F$ imply $H\in \Delta$. Elements of $\Delta$ are called {\it faces}, and maximal faces under inclusion are {\it facets}. We denote the set of facets of $\Delta$ by $\mathcal{F}(\Delta)$. For a face $F \in \Delta$, the {\it link} of $F$ in $\Delta$ is the simplicial complex $\link_{\Delta}F=\{G\subseteq[n]\setminus F\,\,:G\cup F\in\Delta\}$. The {\it Stanley--Reisner ideal} of a simplicial complex $\Delta$ is the ideal, denoted by $I_{\Delta}$, generated by monomials corresponding to the non-faces of $\Delta$, that is, 
$$I_{\Delta}=(x_{i_{1}}\cdots x_{i_{p}}\,\,:\{i_{1}\dots,i_{p}\}\notin\Delta)$$
The ring $K[\Delta]=S/I_\Delta$ is called the {\it Stanley–Reisner ring} of $\Delta$. 

\subsection{Monomial ideals and homological invariants}
Let $I$ be a monomial ideal of $S$. Then we denote by $\mathcal{G}(I)$ the set of minimal monomial generators of $I$. For a subset $F\subset[n]$, we set ${\bf x}_{F}=\prod_{i\in F}x_{i}$, where we set ${\bf x}_{\emptyset}=1$. 

Let $\mathfrak{m}=(x_{1},\ldots, x_{n})$ be the homogeneous maximal ideal, and let $M$ be a finitely generated graded $S$-module. We denote by $\pd M,\depth M$, and $\reg M$ the projective dimension, the depth, and the (Castelnuovo--Mumford) regularity of $M$, respectively. For $i \geq 0$, we write $H^i_{\mathfrak{m}}(M)$ for the $i$-th {\it local cohomology module} of $M$ with respect to $\mathfrak{m}$. The depth and the regularity of $M$ can be described in terms of local cohomology 
$$\depth(M)=\min\{i\,\,:H^i_{\mathfrak{m}}(M)\neq 0\}\mbox{ and }\reg(M)=\max\{i+j\,\,: H^{i}_{\mathfrak{m}}(M)_{j}\neq 0\}.$$

\subsection{Symbolic powers}
Let $I$ be an ideal of $S$. We denote by $\Ass(S/I)$ the set of associated prime ideals of $S/I$, and by $\Min(I)$ the set of minimal prime ideals of $I$. For $k\geq1$, the $k$-th {\it symbolic power} of $I$ is defined by $$I^{(k)}=\bigcap_{P\in\Ass(S/I)}(I^k S_P \cap S).$$
If $I$ is a squarefree monomial ideal with $\Min(I)=\{P_{1},\ldots, P_{r}\}$, then it is known that $$I^{(k)}=P_{1}^{k}\cap\cdots\cap P_{r}^{k}.$$
For a vector $\mathbf{a} = (a_1,\dots,a_n) \in \mathbb{Z}^n$, we set ${\textbf x}^{\mathbf{a}} = x_1^{a_1}\cdots x_n^{a_n}$, $\supp({\bf a})=\{i\,\,:a_{i}\neq0\}$, and $\supp_{-}{\bf a}=\{i\,\,: a_{i}<0\}$. For a subset $F \subseteq [n]$, let $S_F = S[x_j^{-1} : j \in F]$. The simplicial complex $\Delta_{\mathbf{a}}(I)$ is defined as
$$\Delta_{\mathbf{a}}(I)=\{F\setminus\supp_{-}{\bf a}\,\,:\supp_{-}{\bf a} \subseteq F \subseteq [n], \ {\textbf x}^{\mathbf{a}} \notin I S_{F}\}$$
Then it is known that the combinatorial description of the local cohomology module $H^i_{\mathfrak{m}}(S/I)_{\mathbf{a}}$ for monomial ideals by Takayama.  
\begin{thm}[{\normalfont\cite[Theorem 1]{t2005}}]
\label{thm:Takayama}
Let $I\subset S$ be a monomial ideal. Then for every $\mathbf{a} \in \mathbb{Z}^n$ and every $i \geq 0$,
$$H^{i}_{\mathfrak{m}}(S/I)_{\bf a}\simeq\widetilde{H}^{i-|\supp_-({\mathbf{a}})|-1}(\Delta_{\bf a}(I);K).$$
\end{thm}


\section{Minh's conjecture for simplicial graphs}\label{reg}

In this section, we give a regularity formula for edge ideals of simplicial graphs as a partial result of Minh's conjecture posed in \cite[Conjecture A]{mnptv2022}. To prove this, we first prove that Minh's conjecture holds for edge ideals of whisker graphs. After that, we give an inequality for regularities of symbolic powers of squarefree monomial ideals under certain conditions using lcm-lattices. Applying this inequality to a suitable multi-whisker graph, we prove the main theorem of this section. We also study the relationships between the local cohomology modules of ordinary and symbolic powers of edge ideals. 
 
To give a formula for the regularity of symbolic powers of edge ideals of simplicial graph, we first determine the regularity of symbolic powers of edge ideals of whisker graphs. 

\begin{thm}\label{whisker}
For a graph $G$, the whisker graph $W(G)$ satisfies 
$$\reg I(W(G))^{(k)}=\reg I(W(G))^{k}=2k+\im(W(G))-1\mbox{ for all }k\geq1.$$
\end{thm}

To prove Theorem \ref{whisker}, we prepare several lemmas. 
Let $G$ be a graph on the vertex set $X_{[m]}=\{x_{1},\ldots, x_{m}\}$, and let $W(G)$ be the whisker graph on the vertex set $X_{[m]}\cup Y_{[m]}$, where $Y_{[m]}=\{y_{1},\ldots, y_{m}\}$, and whose edge set is
$$E(W(G))=E(G)\cup\{\{x_{i},y_{i}\}\,\,: 1\leq i\leq m\}.$$
Set polynomial rings $S=K[V(G)]$ and $T=K[V(W(G))]$.
Also, set $Q_{r}=(x_{i}y_{i}\,\,: 1\leq i\leq r)$ for each $1\leq r\leq m$ and set $Q_{0}=(0)$.  

\begin{lemma}\label{lemma1}
With the notation introduced, we have 
$$I(W(G))^{(k)}+Q_{m}=I(G)^{(k)}T+Q_{m}\mbox{ for all }k\geq1.$$
\end{lemma}
\begin{proof}
Fix $u={\bf x}^{\bf a}{\bf y}^{\bf b}\in I(W(G))^{(k)}+Q_{m}.$ We may assume that $u\notin Q_{m}$, which implies that $a_{i}b_{i}=0$ for each $i$. Then we have 
$$u\in I(W(G))^{(k)}=\bigcap_{A\in\Delta(G)}P_{A}^{k},$$
where $P_{A}=(x_{i}\,\,: i\notin A)+(y_{j}\,\,: j\in A)$. Hence, we see that $\sum_{i\notin A}a_{i}+\sum_{i\in A}b_{i}\geq k$ for every $A\in\Delta(G)$. Fix $F\in\mathcal{F}(\Delta(G))$ and set $B=\{i\,\,: b_{i}>0\}$. Then, since $F\setminus B\in\Delta(G)$, we have $\sum_{i\notin F\setminus B}a_{i}+\sum_{i\in F\setminus B}b_{i}\geq k$. From the fact that $a_{i}b_{i}=0$ for each $i$, we see that $\sum_{i\notin F}a_{i}\geq k$. Therefore, we obtain ${\bf x}^{\bf a}\in I(G)^{(k)}$, which implies $u\in I(G)^{(k)}T$, as desired. Conversely, fix $v={\bf x}^{\bf a}{\bf y}^{\bf b}\in I(G)^{(k)}T+Q_{m}$. We may assume that $v\notin Q_{m}$, which implies that $a_{i}b_{i}=0$ for each $i$. Then we have 
$${\bf x}^{\bf a}\in I(G)^{(k)}=\bigcap_{F\in\mathcal{F}(\Delta(G))}(x_{i}\,\,: i\notin F)^{k}.$$Fix $A\in\Delta(G)$. Then there exists a facet $F\in\mathcal{F}(\Delta(G))$ such that $A\subset F$. Since $\sum_{i\notin F}a_{i}\geq k$, we obtain 
$$\sum_{i\notin A}a_{i}+\sum_{i\in A}b_{i}\geq\sum_{i\notin A}a_{i}\geq\sum_{i\notin F}a_{i}\geq k,$$
which implies that $v\in I(W(G))^{(k)}$, as required. 
\end{proof}

\begin{lemma}\label{lemma2}
Let $S=K[x_{1},\ldots, x_{m}]$, $T=K[x_{1},\ldots, x_{m}, y_{1},\ldots, y_{m}]$ be polynomial rings over a field $K$ and let $I$ be a monomial ideal of $S$. Then we have  
$$\reg T/(IT+Q_{m})=\max_{F\subset[m]}\{\reg S/(I\,\,: {\bf x}_{F})+|F|\},$$
where we set $\reg 0=-\infty$.  
\end{lemma}
\begin{proof}
Set $T_{r}=K[x_{1},\ldots, x_{m},y_{1},\ldots, y_{r}]$ and $I_{r}=IT_{r}+Q_{r}$, where we set $T_{0}=S$ and $I_{0}=I$. We prove that 
$$\reg T_{r}/I_{r}=\max\{\reg T_{r-1}/I_{r-1}, \reg T_{r-1}/((I:x_{r})T_{r-1}+Q_{r-1})+1\}\mbox{ for each }r\geq1.$$
Suppose that $x_{r}\in I$. Since $x_{r}y_{r}\in IT_{r}$, which implies that $I_{r}=I_{r-1}T_{r}$, and hence, we get $\reg T_{r}/I_{r}=\reg T_{r-1}/I_{r-1}$. Also, since $I:x_{r}=S$, we have $\reg T_{r-1}/((I:x_{r})T_{r-1}+Q_{r-1})=-\infty$, which gives us desired equality. Hence, we suppose that $x_{r}\notin I$. Then we have $I_{r}=I_{r-1}T_{r}+(x_{r}y_{r})$. Notice that $\mathcal{G}(I_{r})=\mathcal{G}(I_{r-1}T_{r})\sqcup\{x_{r}y_{r}\}$ since $x_{r}y_{r}\notin I_{r-1}T_{r}$. Since $I_{r}=I_{r-1}T_{r}+(x_{r}y_{r})$ is the $y_{r}$-partition and the ideal $(x_{r}y_{r})$ has a linear resolution, from \cite[Corollary 2.7]{fhv2009}, we see that $I_{r}=I_{r-1}T_{r}+(x_{r}y_{r})$ is a Betti splitting. Notice that $I_{r-1}T_{r}:y_{r}=I_{r-1}T_{r}$ and $I_{r-1}T_{r}:x_{r}=(I:x_{r})T_{r}+Q_{r-1}$. From this we see that 
\begin{align*}
I_{r-1}T_{r}\cap(x_{r}y_{r})&=x_{r}y_{r}(I_{r-1}T_{r}:x_{r}y_{r})\\
&=x_{r}y_{r}(I_{r-1}T_{r}:x_{r})\\
&=x_{r}y_{r}((I:x_{r})T_{r}+Q_{r-1}T_{r})\\
&=x_{r}y_{r}((I:x_{r})T_{r-1}+Q_{r-1})T_{r}.
\end{align*}
Hence, we have $\reg I_{r-1}T_{r}\cap(x_{r}y_{r})=\reg((I:x_{r})T_{r-1}+Q_{r-1})+2$. Also, we have $\reg T_{r-1}/((I:x_{r})T_{r-1}+Q_{r-1})+1\geq1$, and thus, from  \cite[Corollary 2.2]{fhv2009}, we obtain that 
$$\reg T_{r}/I_{r}=\max\{\reg T_{r-1}/I_{r-1}, \reg T_{r-1}/((I:x_{r})T_{r-1}+Q_{r-1})+1\}.$$
Finally, we prove the equality holds for any monomial ideal using induction on $r$. 
If $r=0$, then the desired equality trivially holds. Hence, we suppose that $r>0$. By the induction hypothesis, we have $\reg T_{r-1}/I_{r-1}=\max_{F\subset[r-1]}\{\reg S/(I:{\bf x}_{F})+|F|\}$ and $\reg T_{r-1}/((I:x_{r})T_{r-1}+Q_{r-1})=\max_{F\subset[r-1]}\{\reg S/(I:{\bf x}_{F\cup\{r\}})+|F|\}$. Therefore, we obtain that 
$$\reg T_{r-1}/((I:x_{r})T_{r-1}+Q_{r-1})+1=\max_{\substack{H\subset[r],\\r\in H}}\{\reg S/(I:{\bf x}_{H})+|H|\},$$ which implies that $\reg T_{r}/I_{r}=\max_{H\subset[r]}\{\reg S/(I:{\bf x}_{H})+|H|\}$, 
as required. 
\end{proof}

\begin{lemma}\label{lemma3}
Let $\Delta$ be a simplicial complex with $I_{\Delta}\neq(0)$ and let $\delta(I_{\Delta})=\lim_{k\rightarrow\infty}\reg I_{\Delta}^{(k)}/k$. Then, for any $k\geq1$ and $F\subset[n]$, we have 
$$\reg S/(I_{\Delta}^{(k)}:{\bf x}_{F})+|F|\leq\delta(I_{\Delta})(k-1)+\dim S/I_{\Delta}.$$
\end{lemma}
\begin{proof}
Fix $k\geq1$ and $F\subset[n]$. We may assume that $S\neq I_{\Delta}^{(k)}:{\bf x}_{F}$. Now, there exist $i\geq0$ and ${\bf a}=(a_{1},\ldots, a_{n})\in\mathbb{Z}^{n}$ such that $H_{\mathfrak{m}}^{i}(S/(I_{\Delta}^{(k)}:{\bf x}_{F}))_{\bf a}\neq0$ and $\reg S/(I_{\Delta}^{(k)}:{\bf x}_{F})=|{\bf a}|+i$, where $|{\bf a}|=a_{1}+\cdots+a_{n}$. Set $\supp_{-}{\bf a}=\{i\,\,: a_{i}<0\}$. Define a vector ${\bf a^{\prime}}\in\mathbb{Z}^{n}$ as $a_{i}^{\prime}=-1$ for each $i\in\supp_{-}{\bf a}$ and $a_{i}^{\prime}=a_{i}$ otherwise. Then we see that $\Delta_{\bf a}(I_{\Delta}^{(k)}:{\bf x}_{F})=\Delta_{\bf a^{\prime}}(I_{\Delta}^{(k)}:{\bf x}_{F})$. Also, from Takayama's formula (\cite[Theorem 1]{t2005}), we obtain $$H_{\mathfrak{m}}^{i}(S/(I_{\Delta}^{(k)}:{\bf x}_{F}))_{\bf a^{\prime}}\simeq\widetilde{H}^{i-|\supp_{-}{\bf a}|-1}(\Delta_{\bf a^{\prime}}(I_{\Delta}^{(k)}:{\bf x}_{F});K)\simeq H_{\mathfrak{m}}^{i}(S/(I_{\Delta}^{(k)}:{\bf x}_{F}))_{\bf a}\neq0,$$which implies that $|{\bf a^{\prime}}|+i\leq\reg S/(I_{\Delta}^{(k)}:{\bf x}_{F})=|{\bf a}|+i$, and hence, we get $|{\bf a^{\prime}}|=|{\bf a}|$. Therefore, we may assume that $a_{i}=-1$ if $i\in\supp_{-}{\bf a}$. Notice that $\supp_{-}{\bf a^{\prime}}\neq[n]$. Set $U=[n]\setminus\supp_{-}{\bf a}$. Define a vector ${\bf b}\in\mathbb{N}^{U}$ as $b_{i}=a_{i}+1$ if $i\in F\setminus\supp_{-}{\bf a}$ and $b_{i}=a_{i}$ otherwise. By definition, $H\in\Delta_{\bf a}(I_{\Delta}^{(k)}:{\bf x}_{F})$ if and only if ${\bf x}^{\bf a}\notin (I_{\Delta}^{(k)}:{\bf x}_{F})S_{(\supp_{-}{\bf a})\cup H}$. Also, the latter one is equivalent to ${\bf x}^{\bf a}{\bf x}_{F}\notin I_{\Delta}^{(k)}S_{(\supp_{-}{\bf a})\cup H}$ since $(I_{\Delta}^{(k)}:{\bf x}_{F})S_{(\supp_{-}{\bf a})\cup H}=I_{\Delta}^{(k)}S_{(\supp_{-}{\bf a)}\cup H}:{\bf x}_{F}$. By the definition of ${\bf b}$, we see that ${\bf x}^{\bf a}{\bf x}_{F}\notin I_{\Delta}^{(k)}S_{(\supp_{-}{\bf a})\cup H}$ is equivalent to ${\bf x}^{\bf b}\notin I_{\Delta}^{(k)}S_{(\supp_{-}{\bf a})\cup H}$. From the proof in \cite[Theorem 3.2]{kmt2025}, we have $\Delta_{\bf c}(I_{\Delta}^{(k)})=\Delta_{\bf c^{+}}(I_{\link_{\Delta}\supp_{-}{\bf c}}^{(k)})$ for any ${\bf c}\in\mathbb{Z}^{n}$, and hence, we see that ${\bf x}^{\bf b}\notin I_{\Delta}^{(k)}S_{(\supp_{-}{\bf a})\cup H}$ is equivalent to ${\bf x}^{\bf b}\notin I_{\link_{\Delta}{\supp_{-}{\bf a}}}^{(k)}R_{H}$, where $R=K[x_{i}\,\,: i\in U]$ and $R_{H}=R[x_{i}^{-1}\,\,: i\in H]$. Hence, we get $\Delta_{\bf a}(I_{\Delta}^{(k)}:{\bf x}_{F})=\Delta_{\bf b}(I_{\link_{\Delta}\supp_{-}{\bf a}}^{(k)})$. 
From Takayama's formula (\cite[Theorem 1]{t2005}), we obtain that 
$$0\neq H_{\mathfrak{m}}^{i}(S/(I_{\Delta}^{(k)}:{\bf x}_{F}))_{\bf a}\simeq\widetilde{H}^{i-|{\supp_{-}{\bf a}}|-1}(\Delta_{\bf b}(I_{\link_{\Delta}\supp_{-}{\bf a}}^{(k)});K)\simeq H_{\mathfrak{n}}^{i-|\supp_{-}{\bf a|}}(R/I_{\link_{\Delta}\supp_{-}{\bf a}}^{(k)}))_{\bf b},$$where $\mathfrak{n}=(x_{i}\,\,: i\in U)$. 
Set $p=i-|\supp_{-}{\bf a}|$. Since $H_{\mathfrak{n}}^{i-|\supp_{-}{\bf a|}}(R/I_{\link_{\Delta}\supp_{-}{\bf a}}^{(k)}))_{\bf b}\neq0$ and \cite[Lemma 2.1]{ht2023} and \cite[Theorem 2.2]{ht2023}, we obtain that $$|{\bf b}|\leq \max\{j\,\,:H_{\mathfrak{n}}^{p}(R/I_{\link_{\Delta}\supp_{-}{\bf a}}^{(k)})_{j}\neq0\}\leq\delta(I_{\link_{\Delta}\supp_{-}{\bf a}})(k-1)\leq\delta(I_{\Delta})(k-1).$$ 
Since $\dim R/I_{\link_{\Delta}\supp_{-}{\bf a}}\leq\dim S/I_{\Delta}-|\supp_{-}{\bf a}|$ and $|{\bf a}|=\sum_{i\in U}a_{i}-|\supp_{-}{\bf a}|$, we obtain that 
\begin{align*}
\reg S/(I_{\Delta}^{(k)}:{\bf x}_{F})+|F|&=|{\bf a}|+i+|F|\\
&=\sum_{i\in U}a_{i}+p+|F|\\
&=|{\bf b}|+p+|F\cap\supp_{-}{\bf a}|\\
&\leq |{\bf b}|+(\dim R/I_{\link_{\Delta}\supp_{-}{\bf a}}^{(k)}+|\supp_{-}{\bf a}|) \\
&\leq\delta(I_{\Delta})(k-1)+\dim S/I_{\Delta},
\end{align*}
which completes the proof. 
\end{proof}

We are now ready to prove Theorem \ref{whisker}. 

\vspace{0.5cm}
\noindent{\it Proof of Theorem }\ref{whisker}. 
From \cite[Corollary 5.4 (5)]{js2021}, it is known that $\reg I(W(G))^{k}=2k+\im(W(G))-1$, and hence, we prove that $\reg I(W(G))^{(k)}=2k+\im(W(G))-1$. Moreover, from \cite[Theorem 4.6]{ghrs2020}, we have $\reg I(W(G))^{(k)}\geq2k+\im(W(G))-1$. Notice that $\im(W(G))=\alpha(G)$ (see, for example \cite[Corollary 4.4]{mpt2025}) and the set of minimal vertex covers of $W(G)$ is $$\Min(W(G))=\{\{x_{i}\,\,: i\notin A\}\cup\{y_{i}\,\,: i\in A\}\,\,: A\in\Delta(G)\}.$$We claim that $\reg T/(I(W(G))^{(k)}+Q_{j})\leq2(k-1)+\alpha(G)$ for all $k\geq1$ and $0\leq j\leq m$. We prove induction on $k\geq1$. Suppose that $k=1$. Since $x_{i}y_{i}$ is a generator of $I(W(G))$, we have $Q_{j}\subset I(W(G))$. Hence, we obtain $T/(I(W(G))+Q_{j})=T/I(W(G))$, and thus, we see that $\reg T/(I(W(G))+Q_{j})=\alpha(G)=2(1-1)+\alpha(G)$. Now, let $k>1$ and assume that $\reg T/(I(W(G))^{(k-1)}+Q_{j})\leq2(k-2)+\alpha(G)$ holds for all $j$. We first prove for $j=m$. By Lemma \ref{lemma1}, we have $I(W(G))^{(k)}+Q_{m}=I(G)^{(k)}T+Q_{m}$. Hence, from Lemma \ref{lemma2}, we obtain that 
$$\reg T/(I(W(G))^{(k)}+Q_{m})=\max_{F\subset[m]}\{\reg S/(I(G)^{(k)}:{\bf x}_{F})+|F|\}.$$
Also, from Lemma \ref{lemma3} and \cite[Example 4.4]{dhnt2021}, we get $$\reg S/(I(G)^{(k)}:{\bf x}_{F})+|F|\leq\delta(I(G))(k-1)+\alpha(G)=2(k-1)+\alpha(G).$$
Therefore, we obtain that 
$$\reg T/(I(W(G))^{(k)}+Q_{m})\leq 2(k-1)+\alpha(G).$$
We now prove the assertion for $j-1$, assuming that it holds for $j$, where $1\leq j\leq m$. 
Since $Q_{j-1}:x_{j}y_{j}=Q_{j-1}$ and the fact that $I(W(G))^{(k)}\,\,: x_{j}y_{j}=I(W(G))^{(k-1)}$ given in the proof in \cite[Theorem 5.2]{kty2018}, we obtain 
$(I(W(G))^{(k)}+Q_{j-1}):x_{j}y_{j}=I(W(G))^{(k-1)}+Q_{j-1}$. Hence, we get a short exact sequence 
$$0\rightarrow T/(I(W(G))^{(k-1)}+Q_{j-1})(-2)\rightarrow T/(I(W(G))^{(k)}+Q_{j-1})\rightarrow T/(I(W(G))^{(k)}+Q_{j})\rightarrow0,$$
which implies that $$\reg T/(I(W(G))^{(k)}+Q_{j-1})\leq\max\{\reg T/(I(W(G))^{(k-1)}+Q_{j-1})+2, \reg T/(I(W(G))^{(k)}+Q_{j})\}.$$
By the induction hypothesis on $k$ and the descending induction hypothesis on $j$, we have $\reg T/(I(W(G))^{(k-1)}+Q_{j-1})+2\leq2(k-1)+\alpha(G)$ and $\reg T/(I(W(G))^{(k)}+Q_{j})\leq 2(k-1)+\alpha(G)$. Therefore, we obtain $\reg T/(I(W(G))^{(k)}+Q_{j})\leq 2(k-1)+\alpha(G)$ for all $k\geq1$ and $0\leq j\leq m$. In particular, taking $j=0$, we have $\reg I(W(G))^{(k)}\leq 2(k-1)+\alpha(G)+1$, as required. 
$\hspace{14.4cm}\square$

In what follows, we prove that Minh's conjecture holds for simplicial graphs as the main theorem of this section. Let us recall the definition of simplicial graphs. Let $G$ be a graph. A vertex $v\in V(G)$ is called {\it simplicial} if $N_{G}[v]$ is a clique in $G$. $G$ is called {\it simplicial} if every vertex of $G$ is contained in the closed neighborhood of some simplicial vertex. To prove the main theorem in this section, we first prove that Minh's conjecture holds for multi-whisker graphs, which was introduced by Pournaki, Terai, and the author in \cite{mpt2025} as a generalization of whisker graphs using Theorem \ref{whisker} and \cite{bh2014}. 

\begin{defi}[\cite{bh2014}]
Let $I=({\bf x}^{\bf a_{1}},\ldots,{\bf x}^{\bf a_{m}})\subset S$ be a monomial ideal and $(i_{1},\ldots, i_{n})$ an $n$-tuple with positive integer entries. Let $S^{*}=K[x_{j,\ell}\,\,:j\in[n], \ell\in[i_{j}]]$ be a polynomial ring over $K$ and $P_{j}=(x_{j,1},\ldots, x_{j,i_{j}})$ be the monomial prime ideal for $1\leq j\leq n$. Then the {\it expansion} of $I$ with respect to $(i_{1},\ldots, i_{n})$ defined as the monomial ideal 
$$I^{*}=\sum_{1\leq i\leq m}P_{1}^{a_{i}(1)}\cdots P_{n}^{a_{i}(n)}\subset S^{*},$$where ${\bf a}_{i}=(a_{i}(1),\ldots, a_{i}(n))$.
\end{defi}

It is known the following results. 

\begin{lemma}[\cite{bh2014} Lemma 1.1(iii), Corollary 1.4, Theorem 4.2]\label{expansion lemma}
Let $I$ be monomial ideals. Then we have 
\begin{enumerate}
\item $(I^{k})^{*}=(I^{*})^{k}$ for all $k\geq1$. 
\item $(I^{(k)})^{*}=(I^{*})^{(k)}$ for all $k\geq1$. 
\item $\reg I=\reg I^{*}$
\end{enumerate}
\end{lemma}

Let us recall the definition of multi-whisker graphs. Given a graph $G_0$ on the vertex set  $X_{[h]}=\{x_{1}, \ldots, x_{h}\}$ and positive integers $n_{1},\ldots, n_{h}$, the \textit{multi-whisker graph associated with }$G_{0}$ is $G=G_{0}[n_{1},\ldots, n_{h}]$ on the vertex set 
$$X_{[h]}\cup Y, \mbox{ where } Y=\{y_{1,1},\ldots, y_{1, n_{1}}\}\cup\cdots\cup\{y_{h,1},\ldots, y_{h,n_{h}}\}$$ and the edge set $$E(G)=E(G_{0})\cup\{\{x_1,y_{1,1}\},\ldots, \{x_{1},y_{1,n_{1}}\},\ldots, \{x_{h},y_{h,1}\},\ldots, \{x_{h},y_{h,n_{h}}\}\}.$$

\begin{cor}\label{multi-whisker}
With the notation introduced, we have  
$$\reg I(G)^{(k)}=\reg I(G)^{k}=2k+\im(G)-1\mbox{ for all }k\geq1.$$
\end{cor}
\begin{proof}
 We can verify that $I(G)=I(W(G_{0}))^{*}$. From Theorem \ref{whisker} and Lemma \ref{expansion lemma}, we have 
$$\reg I(G)^{(k)}=\reg I(W(G_{0}))^{(k)}=\reg I(W(G_{0}))^{k}=\reg I(G)^{k}\mbox{ for all }k\geq1.$$
Moreover, we know $\im(G)=\alpha(G_{0})=\im(W(G_{0}))$, as required. 
\end{proof}

Next, we prepare the key lemma in order to prove the main theorem. As a notation, for a squarefree monomial ideal $I\subset K[x_{1},\ldots, x_{n}]$, we set 
$\mathcal{C}(I)=\{C\subset[n]\,\,: (x_{i}\,\,: i\in C)\in\Min(I)\}$ and write $I^{\vee}$ for the Alexander dual ideal of $I$. Also, let us recall the definition of lcm-lattices of monomial ideals introduced in \cite{gpw1999}. Let $I\subset S$ be a monomial ideal. The {\it lcm-lattice} of $I$, denoted by $L_{I}$, is the set consisting of 1 and all least common multiples of subsets of $\mathcal{G}(I)$, partially ordered by divisibility. 

\begin{lemma}\label{key}
Let $I\subset K[x_{1},\ldots, x_{n}]$ and $J\subset K[y_{1},\ldots, y_{m}]$ be squarefree monomial ideals. Assume that there exists a map $\varphi:[n]\rightarrow[m]$ satisfies the following conditions:
\begin{enumerate}
\item $\varphi|_{C}$ is injective for every $C\in\mathcal{C}(I)$. 
\item $\mathcal{C}(J)=\{\varphi(C)\,\,: C\in\mathcal{C}(I)\}$.
\end{enumerate}
Then we have $\reg J^{(k)}\leq\reg I^{(k)}\mbox{ for all }k\geq1.$
\end{lemma}
\begin{proof}
Fix $k\geq1$ and set $S=K[x_{1},\ldots, x_{n}]$, $T=K[y_{1},\ldots, y_{m}]$, and set $P_{C}=(x_{i}\,\,: i\in C)$ for $C\in\mathcal{C}(I)$. Notice that $I^{(k)}=\bigcap_{C\in\mathcal{C}(I)}P_{C}^{k}$. Write $S^{\rm pol}=K[x_{i,r}\,\,: i\in[n], r\in[k]]$ and $T^{\rm pol}=K[y_{j,r}\,\,: j\in[m],r\in[k]]$. From \cite[Proposition 2.3]{f2006}, we know $(I^{(k)})^{\rm pol}=\bigcap_{C\in\mathcal{C}(I)}(P_{C}^{k})^{\rm pol}$. Moreover, we see that 
$$(I^{(k)})^{\rm pol}=\bigcap_{C\in\mathcal{C}(I)}\bigcap_{\substack{1\leq c_{i}\leq k\\ \sum_{i\in C}c_{i}\leq k+|C|-1}}(x_{i,c_{i}}\,\,: i\in C),$$
which implies that 
$$\mathcal{G}(((I^{(k)})^{\rm pol})^{\vee})=\left\{\prod_{i\in C}x_{i,a_{i}+1}\,\,: C\in\mathcal{C}(I), a_{i}\geq0,\sum_{i\in C}a_{i}\leq k-1\right\}.$$
Define a map $\tilde{\varphi}:[n]\times[k]\rightarrow[m]\times[k]$ by $\tilde{\varphi}(i,r)=(\varphi(i), r)$ for each $(i,r)\in[n]\times[k]$. For a squarefree monomial $u\in S^{\rm pol}$, we set $\supp(u)=\{(i,r)\,\,: x_{i,r}\mbox{ divides }u\}$ and $\Phi(u)=\prod_{(j,r)\in\tilde{\varphi}(\supp(u))}y_{j,r}$. Let $L_{((I^{(k)})^{\rm pol})^{\vee}}$ and $L_{((J^{(k)})^{\rm pol})^{\vee}}$ denote the lcm-lattices of $((I^{(k)})^{\rm pol})^{\vee}$ and $((J^{(k)})^{\rm pol})^{\vee}$, respectively. 
Define a map $\delta:L_{((I^{(k)})^{\rm pol})^{\vee}}\rightarrow L_{((J^{(k)})^{\rm pol})^{\vee}}$ by $\delta(u)=\Phi(u)$. First we verify that $\delta$ is well-defined. Fix $u\in L_{((I^{(k)})^{\rm pol})^{\vee}}$. By definition, there exist squarefree monomials $u_{1},\ldots,u_{s}$ in $\mathcal{G}(((I^{(k)})^{\rm pol})^{\vee})$ such that $u=\lcm(u_{1},\ldots, u_{s})$. Since $u_{i}\in \mathcal{G}(((I^{(k)})^{\rm pol})^{\vee})$, we can write $u_{\ell}=\prod_{i\in C}x_{i,a_{i}+1}$ for some $C\in\mathcal{C}(I)$ and $a_{i}\geq0$ such that $\sum_{i\in C}a_{i}\leq k-1$. Since $\supp(u_{\ell})=\{(i,a_{i}+1)\,\,: i\in C\}$, we obtain $\tilde{\varphi}(\supp(u_{\ell}))=\{(\varphi(i), a_{i}+1)\,\,: i\in C\}$. By the assumption that $\varphi|_{C}$ is injective, we see that $\Phi(u_{\ell})=\prod_{i\in C}y_{\varphi(i),a_{i}+1}$. Also, by the assumption that $\mathcal{C}(J)=\{\varphi(C)\,\,: C\in\mathcal{C}(I)\}$, we get $\varphi(C)\in\mathcal{C}(J)$. Moreover, since $\varphi|_{C}:C\rightarrow\varphi(C)$ is surjective, for any $j\in\varphi(C)$, there exits the unique $i\in C$ such that $j=\varphi(i)$, and we set $b_{j}=a_{i}$. Then we have $\sum_{j\in\varphi(C)}b_{j}=\sum_{i\in C}a_{i}\leq k-1$ and $\Phi(u_{\ell})=\prod_{j\in\varphi(C)}y_{j, b_{j}+1}$, which imply that $\Phi(u_{\ell})\in\mathcal{G}(((J^{(k)})^{\rm pol})^{\vee})$. 
Notice that we have $\supp(\lcm(u_{1},\ldots, u_{s}))=\supp(u_{1})\cup\cdots\cup\supp(u_{s})$. Hence, we see that 
\begin{align*}
\supp(\delta(u))&=\supp(\Phi(u)) \\
&=\supp\bigl(\Phi(\lcm(u_{1},\ldots,u_{s}))\bigr) \\
&=\tilde{\varphi}\bigl(
  \supp(u_{1})\cup\cdots\cup\supp(u_{s})
  \bigr) \\
&=\tilde{\varphi}\bigl(\supp(u_{1})\bigr)
  \cup\cdots\cup
  \tilde{\varphi}\bigl(\supp(u_{s})\bigr) \\
&=\supp\bigl(
  \lcm(\Phi(u_{1}),\ldots,\Phi(u_{s}))
  \bigr),
\end{align*}
which implies that $\delta(u)=\lcm(\Phi(u_{1}),\ldots,\Phi(u_{s}))$ since both are squarefree monomials. Hence, we obtain $\delta(u)\in L_{((J^{(k)})^{\rm pol})^{\vee}}$. 
Next, we prove that $\delta$ is a join-preserving map. Fix $u, v\in L_{((I^{(k)})^{\rm pol})^{\vee}}$. Then we have 
$$\delta(u\vee v)=\Phi(\lcm(u,v))=\lcm(\Phi(u), \Phi(v))=\lcm(\delta(u),\delta(v))=\delta(u)\vee\delta(v),$$which implies that $\delta$ is a join-preserving map. Finally, we prove that $\delta$ is surjective and $\delta^{-1}(1)=\{1\}$. Fix $v\in L_{((J^{(k)})^{\rm pol})^{\vee}}$. Then there exist squarefree monomial $v_{1},\ldots, v_{t}\in\mathcal{G}(((J^{(k)})^{\rm pol})^{\vee})$ such that $v=\lcm(v_{1},\ldots, v_{t})$. Fix $1\leq\ell\leq t$. Then we can write $v_{\ell}=\prod_{j\in D_{\ell}}y_{j,b_{\ell,j}+1}$ for some $D_{\ell}\in\mathcal{C}(J)$ such that $\sum_{j\in D_{\ell}}b_{\ell,j}\leq k-1$. By the assumption that $\mathcal{C}(J)=\{\varphi(C)\,\,: C\in\mathcal{C}(I)\}$, we can take $C_{\ell}\in\mathcal{C}(I)$ such that $\varphi(C_{\ell})=D_{\ell}$. Now, we have $\varphi|_{C_{\ell}}$ is bijective. For every $i\in C_{\ell}$, set $a_{\ell,i}=b_{\ell,\varphi(i)}$ and $u_{\ell}=\prod_{i\in C_{\ell}}x_{i, a_{\ell,i}+1}$. Then since $\sum_{i\in C_{\ell}}a_{\ell,i}=\sum_{j\in D_{\ell}}b_{\ell,j}\leq k-1$, which implies that $u_{\ell}\in\mathcal{G}(((I^{(k)})^{\rm pol})^{\vee})$. Moreover, we have $\delta(u_{\ell})=\Phi(u_{\ell})=\prod_{i\in C_{\ell}}y_{\varphi(i),a_{\ell,i}+1}=v_{\ell}$. By setting $u=\lcm(u_{1},\ldots, u_{t})$, since $\delta$ is join-preserving, we see that $\delta(u)=\delta(\lcm(u_{1},\ldots, u_{t}))=\lcm(v_{1},\ldots, v_{t})=v$, which implies that $\delta$ is surjective. Notice that $\delta(1)=\Phi(1)=1$. Suppose that $u\neq1$. Since $\supp(u)\neq\emptyset$, we obtain that $\supp(\delta(u))=\supp(\Phi(u))=\tilde{\varphi}(\supp(u))\neq\emptyset$. 
Hence, we get $\delta^{-1}(1)=\{1\}$. Consequently, we have $\delta:L_{((I^{(k)})^{\rm pol})^{\vee}}\rightarrow L_{((J^{(k)})^{\rm pol})^{\vee}}$ is a well-defined surjective join-preserving map with $\delta^{-1}(1)=\{1\}$. 
From \cite[Theorem 4.9]{ikmf2017} and \cite[Corollary 0.3]{t1999}, we obtain that 
$$\reg J^{(k)}=\pd T^{\rm pol}/((J^{(k)})^{\rm pol})^{\vee}\leq\pd S^{\rm pol}/((I^{(k)})^{\rm pol})^{\vee}=\reg I^{(k)},$$
which completes the proof. 
\end{proof}

To apply Lemma \ref{key} for simplicial graphs, we give a relationship between a simplicial graph and a certain multi-whisker graph. To this end, we prepare the following lemma.

\begin{lemma}\label{structure of simplicial graphs}
Let $G$ be a graph. Then $G$ is a simplicial graph if and only if there exist a graph $G_{0}$ and a family of cliques $\mathcal{Q}=(Q_{1},\ldots, Q_{r})$ of $G_{0}$ with $V(G_{0})=\bigcup_{1\leq i\leq r}Q_{i}$ and new vertices $w_{1},\ldots, w_{r}$ such that $V(G)=V(G_{0})\cup\{w_{1},\ldots, w_{r}\}$ and 
$$E(G)=E(G_{0})\cup\left(\bigcup_{1\leq i\leq r}\{\{x,w_{i}\}\,\,: x\in Q_{i}\}\right).$$
\end{lemma}
\begin{proof}
Suppose that $G$ is simplicial. Let $U$ denote the set of simplicial vertices and take a maximal independent set $A=\{w_{1},\ldots, w_{r}\}$ of $G|_{U}$ and set $G_{0}=G-A$ and $Q_{i}=N_{G}(w_{i})$. Since $A$ is a maximal independent set of $G|_{U}$ and $w_{i}$ is a simplicial vertex, we see that $Q_{i}\subset V(G_{0})$ and $Q_{i}=N_{G}(w_{i})$ is a clique. Fix $x\in V(G_{0})$. Assume $x$ is a simplicial vertex. Since $x\in U\setminus A$ and $A$ is a maximal independent set, we see that $\{x,w_{i}\}\in E(G)$ for some $i$, which implies that $x\in N_{G}(w_{i})=Q_{i}$. Assume that $x$ is not simplicial. Since $G$ is simplicial, there exists $u\in U$ such that $x\in N_{G}(u)$. If $u\in A$, then $x\in N_{G}(w_{i})$ for some $i$. Suppose that $u\notin A$. By the maximality of $A$, we see that $u\in N_{G}(w_{i})$ for some $i$. Since $x, w_{i}\in N_{G}(u)$ and $u$ is simplicial, we obtain that $\{x,w_{i}\}\in E(G)$, and hence, $x\in N_{G}(w_{i})=Q_{i}$. Therefore, we get $V(G_{0})=\bigcup_{1\leq i\leq r}Q_{i}$. Now, we have $E(G_{0})=\{\{x,y\}\in E(G)\,\,: x,y\notin A\}$. Fix $\{x,y\}\in E(G)$. Since $A$ is an independent set, we have $\{x,y\}\not\subset A$. If $x,y\notin A$, then we have $\{x,y\}\in E(G_{0})$. Hence, we may assume that $x\in A$ and $y\notin A$. Then, we can write $x=w_{i}$ for some $i$. Hence, $\{x,y\}=\{w_{i},y\}\in(\bigcup_{1\leq i\leq r}\{\{x,w_{i}\}\,\,: x\in Q_{i}\})$, and hence, we obtain $E(G)\subset E(G_{0})\cup\left(\bigcup_{1\leq i\leq r}\{\{x,w_{i}\}\,\,: x\in Q_{i}\}\right)$, as desired. 

Conversely, suppose that there exist a graph $G_{0}$ and a family of cliques $\mathcal{Q}=(Q_{1},\ldots, Q_{r})$ of $G_{0}$ with $V(G_{0})=\cup_{1\leq i\leq r}Q_{i}$ and new vertices $w_{1},\ldots, w_{r}$ such that $V(G)=V(G_{0})\cup\{w_{1},\ldots, w_{r}\}$ and 
$$E(G)=E(G_{0})\cup\left(\bigcup_{1\leq i\leq r}\{\{x,w_{i}\}\,\,: x\in Q_{i}\}\right).$$Notice that $N_{G}(w_{i})=Q_{i}$ for every $i$. Since $Q_{i}$ is a clique, $w_{i}$ is a simplicial vertex of $G$. 
Now, let $x\in V(G_{0})$. Since $V(G_{0})=\bigcup_{1\leq i\leq r}Q_{i}$, there exists $i$ such that $x\in Q_{i}$. Then, by assumption, $\{x, w_{i}\}\in E(G)$, and thus, $G$ is simplicial, as required. 
\end{proof}

We now prove the main theorem of this section. 

\begin{thm}\label{main1}
Let $G$ be a simplicial graph with $E(G)\neq\emptyset$. Then we have 
$$\reg I(G)^{(k)}=\reg I(G)^{k}=2k+\im(G)-1\mbox{ for all }k\geq1.$$
\end{thm}
\begin{proof}
Fix $k\geq1$. From Lemma \ref{structure of simplicial graphs},  there exist a graph $G_{0}$ and a family of cliques $\mathcal{Q}=(Q_{1},\ldots, Q_{r})$ of $G_{0}$ with $V(G_{0})=\cup_{1\leq i\leq r}Q_{i}$ and new vertices $w_{1},\ldots, w_{r}$ such that $V(G)=V(G_{0})\cup\{w_{1},\ldots, w_{r}\}$ and $E(G)=E(G_{0})\cup(\bigcup_{1\leq i\leq r}\{\{x,w_{i}\}\,\,: x\in Q_{i}\})$. Set $V(G_{0})=X_{[h]}$. For each $i\in[h]$, we set $\Lambda_{i}=\{j\in[r]\,\,: x_{i}\in Q_{j}\}$ and $n_{i}=|\Lambda_{i}|$. Since $V(G_{0})=\cup_{1\leq j\leq r}Q_{j}$, we have $n_{i}>0$ for all $i$. Define a graph $G^{\prime}$ on the vertex set $V(G^{\prime})=X_{[h]}\cup\{y_{i,j}\,\,: i\in[h], j\in\Lambda_{i}\}$ and whose edge set is 
$$E(G^{\prime})=E(G_{0})\cup\{\{x_{i},y_{i,j}\}\,\,: i\in[h], j\in\Lambda_{i}\}.$$
Then we see that $G^{\prime}=G_{0}[n_{1},\ldots, n_{h}]$ is a multi-whisker graph. Then, we define a map $\varphi:V(G^{\prime})\rightarrow V(G)$ as $\varphi(x_{i})=x_{i}$ and $\varphi(y_{i,j})=w_{j}$ for all $i,j$. Set $C^{\prime}_{A}=(X_{[h]}\setminus A)\cup\{y_{i,j}\,\,: x_{i}\in A, j\in\Lambda_{i}\}$ and $C_{A}=(X_{[h]}\setminus A)\cup\{w_{j}\,\,: A\cap Q_{j}\neq\emptyset\}$ for each $A\in\Delta(G_{0})$. Now, we can easily show that $$\mathcal{C}(I(G^{\prime}))=\{C^{\prime}_{A}\,\,: A\in\Delta(G_{0})\}\mbox{ and }\mathcal{C}(I(G))=\{C_{A}\,\,: A\in\Delta(G_{0})\}.$$
Hence, since $\varphi(C^{\prime}_{A})=C_{A}$ for each $A\in\Delta(G_{0})$, we get $\mathcal{C}(I(G))=\{\varphi(C^{\prime})\,\,: C^{\prime}\in\mathcal{C}(I(G^{\prime}))\}$. Fix $A\in\Delta(G_{0})$ and we verify that $\varphi|_{C^{\prime}_{A}}: C^{\prime}_{A}\rightarrow C_{A}$ is injective. Suppose that $\varphi(u)=\varphi(v)$ for $u,v\in C^{\prime}_{A}$. If $u, v\in X_{[h]}$, then, by the definition of $\varphi$, we have $u=v$. Also, if either $u\in X_{[h]}$ and $v\notin X_{[h]}$ or $v\in X_{[h]}$ and $u\notin X_{[h]}$, then,  by the definition of $\varphi$, we obtain $\varphi(u)\neq\varphi(v)$, a contradiction. Hence, we may assume that $u=y_{i,j}$ and $v=y_{\ell,t}$ for some $i,j,\ell,t$. Then we have $w_{j}=\varphi(u)=\varphi(v)=w_{t}$, which implies that $j=t$. Hence, we can write $u=y_{i,j}$ and $v=y_{\ell,j}$. Also, since $u,v\in C_{A}^{\prime}$, we know $x_{i},x_{\ell}\in A$ and $j\in\Lambda_{i}\cap\Lambda_{\ell}$. Thus, if $i\neq\ell$, then, since $Q_{j}$ is a clique, we obtain $\{x_{i},x_{\ell}\}\in E(G_{0})$, which is a contradiction to the fact that $A\in\Delta(G_{0})$. Therefore, $\varphi|_{C_{A}^{\prime}}$ is injective for any $A\in\Delta(G_{0})$. From Lemma \ref{key}, we obtain that $\reg I(G)^{(k)}\leq\reg I(G^{\prime})^{(k)}$. Now, we can easily check that $\im(G)=\alpha(G_{0})=\im(G^{\prime})$. Since $G^{\prime}$ is a multi-whisker graph, from Corollary \ref{multi-whisker} and \cite[Theorem 4.6]{ghrs2020}, we get $$2k+\im(G)-1\leq\reg I(G)^{(k)}\leq\reg I(G^{\prime})^{(k)}=2k+\im(G^{\prime})-1=2k+\im(G)-1.$$
Notice that, in particular, we obtain $\reg I(G)=\im(G)+1$. Combining this and \cite[Corollary 5.5]{w2011}, \cite[Theorem 5.3]{js2021}, we have $\reg I(G)^{k}\leq 2k+\reg I(G)-2=2k+\im(G)-1$. On the other hand, from \cite[Theorem 4.5]{bht2015}, we know $2k+\im(G)-1\leq\reg I(G)^{k}$. Hence, we get $\reg I(G)^{k}=2k+\im(G)-1$, as required. 
\end{proof}

We end this section with a result on local cohomology of ordinary and symbolic powers of edge ideals. As a result, we also get an inequality for the regularity. 

\begin{lemma}\label{surj}
Let $I\subset J$ be monomial ideals of $S$. Suppose that $\sqrt{I:{\bf x}^{\bf a}}=\sqrt{J:{\bf x}^{\bf a}}$ for every ${\bf a}\in\mathbb{N}^{n}$ such that ${\bf x}^{\bf a}\notin J$. Then, for every $i$, the natural map $H_{\mathfrak{m}}^{i}(S/I)\rightarrow H_{\mathfrak{m}}^{i}(S/J)$ induced by the short exact sequence
$0\rightarrow J/I\rightarrow S/I\rightarrow S/J\rightarrow0$ is surjective. 
\end{lemma}
\begin{proof}
It suffices to prove that $H_{\mathfrak{m}}^{i}(J/I)\rightarrow H_{\mathfrak{m}}^{i}(S/I)$ is injective for all $i$. Fix ${\bf a}\in\mathbb{Z}^{n}$ and set ${\bf a}^{+}=(\max\{a_{1},0\},\ldots, \max\{a_{n},0\})\in\mathbb{N}^{n}$. We claim that $\Delta_{\bf a}(I)=\Delta_{\bf a}(J)$ or $\Delta_{\bf a}(J)=\emptyset$. First, suppose that ${\bf x}^{\bf a^{+}}\notin J$. By the assumption, we have $\sqrt{I:{\bf x}^{\bf a^{+}}}=\sqrt{J:{\bf x}^{\bf a^{+}}}$, and hence, by \cite[Lemma 2.19]{mnptv2022}, we see that $\Delta_{\bf a^{+}}(I)=\Delta_{\bf a^{+}}(J)$. As in the proof of \cite[Theorem 3.2]{kmt2025}, we know $\Delta_{\bf a}(I)=\link_{\Delta_{\bf a^{+}}(I)}\supp_{-}{\bf a}=\link_{\Delta_{\bf a^{+}}(J)}\supp_{-}{\bf a}=\Delta_{\bf a}(J)$. Also, suppose that ${\bf x}^{\bf a^{+}}\in J$. From \cite[Lemma 2.19]{mnptv2022}, we know that $\Delta_{\bf a^{+}}(J)$ is the void complex, and thus, we see that $\Delta_{\bf a}(J)=\link_{\Delta_{\bf a^{+}}(J)}\supp_{-}{\bf a}=\emptyset$. Now, for each ${\bf a}\in\mathbb{Z}^{n}$, from Takayama's formula (\cite[Theorem 1]{t2005}) and relative Hochster--Takayama formula (\cite[Theorem 3.2]{hm2026}), we obtain that 
$$H_{\mathfrak{m}}^{i}(S/I)_{\bf a}\simeq\widetilde{H}^{i-|\supp_{-}{\bf a}|-1}(\Delta_{\bf a}(I);K)\mbox{ and }H_{\mathfrak{m}}^{i}(J/I)_{\bf a}\simeq\widetilde{H}^{i-|\supp_{-}{\bf a}|-1}(\Delta_{\bf a}(I),\Delta_{\bf a}(J);K).$$We distinguish the following cases. 

{\bf Case 1.} Suppose that $\Delta_{\bf a}(I)=\Delta_{\bf a}(J)$. In this case, we know $$H_{\mathfrak{m}}^{i}(J/I)_{\bf a}\simeq\widetilde{H}^{i-|\supp_{-}{\bf a}|-1}(\Delta_{\bf a}(I),\Delta_{\bf a}(I);K)=0,$$ and hence, we see that the induced map $H_{\mathfrak{m}}^{i}(J/I)_{\bf a}\rightarrow H_{\mathfrak{m}}^{i}(S/I)_{\bf a}$ is injective for all $i$. 

{\bf Case 2.} Suppose that $\Delta_{\bf a}(J)=\emptyset$. In this case, we know $$H_{\mathfrak{m}}^{i}(S/J)_{\bf a}\simeq\widetilde{H}^{i-|\supp_{-}{\bf a}|-1}(\Delta_{\bf a}(J);K)=0,$$ and hence, we see that the induced map $H_{\mathfrak{m}}^{i}(J/I)_{\bf a}\rightarrow H_{\mathfrak{m}}^{i}(S/I)_{\bf a}$ is isomorphism for all $i$,
which completes the proof. 
\end{proof}

The {\it odd girth} of a graph is defined as the length of a shortest induced odd cycle in $G$, denoted by $\ogirth(G)$, where we set $\ogirth(G)=\infty$ if $G$ is bipartite. Then it is known that the following result, which shows that when $I(G)^{(k)}=I(G)^{k}$ holds.  

\begin{lemma}[\cite{rty2011}, Lemma 3.10]\label{char of symb=ord}
Let $G$ be a graph and $k\geq2$. Then the following conditions are equivalent:
\begin{enumerate}
\item $I(G)^{(k)}=I(G)^{k}$. 
\item $\ogirth(G)>2k-1$. 
\end{enumerate}
\end{lemma}

\begin{thm}\label{local coho}
Let $G$ be a graph with $\ogirth(G)=2t-1$. If $1\leq k\leq t$, then, for every $i$, the natural map $H_{\mathfrak{m}}^{i}(S/I(G)^{k})\rightarrow H_{\mathfrak{m}}^{i}(S/I(G)^{(k)})$ induced from $0\rightarrow I(G)^{(k)}/I(G)^{k}\rightarrow S/I(G)^{k}\rightarrow S/I(G)^{(k)}\rightarrow0$ is surjective.  
\end{thm}
\begin{proof}
Set $I=I(G)$. For every $1\leq k<t$, from Lemma \ref{char of symb=ord}, we have $I^{(k)}=I^{k}$, and hence $I^{(k)}/I^{k}=(0)$,  and there is nothing to prove. Thus, it remains to consider the case $k=t$. We now prove that $\sqrt{I^{(t)}:{\bf x}^{\bf b}} = \sqrt{I^t:{\bf x}^{\bf b}}$ for every ${\bf b}\in\mathbb{N}^{n}$ with ${\bf x}^{\bf b}\notin I^{(t)}$. First, we claim that $$I^{(t)} = I^{t}+ ({\bf x}_C: C\text{ is an induced cycle of length }2t-1)$$ where ${\bf x}_C=\prod_{x_j\in V(C)}x_j$. Let $u$ be a minimal monomial generator of $I^{(t)}$ with $u\notin I^t$, and set $W=\supp(u)$ and $H=G|_{W}$. Now we have  $u\in I(H)^{(t)}\setminus I(H)^{t}$. Hence, from Lemma \ref{char of symb=ord}, $H$ contains an odd cycle of length at most $2t-1$. Since $H$ is an induced subgraph of $G$ and $\ogirth(G)=2t-1$, the graph $H$ contains an induced cycle $C$ of length exactly $2t-1$. Since $|V(C)\cap D|\geq t$ for every minimal vertex cover of $G$, we see that ${\bf x}_{C}\in I^{(t)}$. Also, since $V(C)\subset W$, ${\bf x}_{C}$ divides $u$, and thus, we see that $u={\bf x}_{C}$.
Therefore, we obtain 
$$I^{(t)}:{\bf x}^{\bf b}=(I^{t}:{\bf x}^{\bf b})+({\bf x}_{C}/\gcd({\bf x}_{C}, {\bf x}^{\bf b})\,\,: C\mbox{ is an induced cycle of length }2t-1).$$
Hence, it is enough to prove that ${\bf x}_C/\gcd({\bf x}_C,x^{\bf b}) \in \sqrt{I^t:{\bf x}^{\bf b}}$ for every induced cycle $C$ of length $2t-1$. Set $f_{C}={\bf x}_{C}/\gcd({\bf x}_{C},{\bf x}^{\bf b})$. Since ${\bf x}_C\in I^{(t)}$ and ${\bf x}^{\bf b}\notin I^{(t)}$, we have $f_C\neq1$. Take a vertex $v\in V(C)$ such that $x_v$ divides $f_C$. Write $C=(v_{1},\ldots,v_{2t-1},v_{1})$ with $v=v_{1}$. Then we see that $x_v{\bf x}_C = (x_{v_{1}}x_{v_{2}})(x_{v_{1}}x_{v_{2t-1}}) \prod_{j=1}^{t-2}(x_{v_{2j+1}}x_{v_{2j+2}}) \in I^t$. Notice that $f_{C}=\prod_{w\in V(C), b_{w}=0}x_{w}$. Hence, we see that $x_{v}{\bf x}_{C}=x_{v}^{2}\prod_{w\in V(C)\setminus\{v\}}x_{w}$ divides $f_{C}^{2}{\bf x}^{\bf b}$. Hence, we get $\sqrt{I^{(t)}:{\bf x}^{\bf b}} = \sqrt{I^t:{\bf x}^{\bf b}}$, and thus, from Lemma \ref{surj}, the assertion follows. 
\end{proof}

As an immediate consequence, we obtain the following partial regularity comparison between ordinary and symbolic powers.

\begin{cor}\label{ineq of reg}
Let $G$ be a graph with $\ogirth(G)=2t-1$. If $1\leq k\leq t$, then we have $$\reg I(G)^{(k)}\leq\reg I(G)^{k}.$$ 
\end{cor}
\begin{proof}
From Lemma \ref{char of symb=ord}, if $1\leq k<t$, then we have $I(G)^{(k)}=I(G)^{k}$. Hence, it is enough to consider the case that $k=t$. From Theorem \ref{local coho}, if $H_{\mathfrak{m}}^{i}(S/I(G)^{(k)})_{j}\neq0$, then we have $H_{\mathfrak{m}}^{i}(S/I(G)^{k})_{j}\neq0$, and thus, the assertion follows. 
\end{proof}


\section{A comparison of depths between ordinary and symbolic powers}\label{depth}

In this section, we investigate a comparison of the depth between ordinary and symbolic powers of edge ideals. For each $k\geq2$, we provide a squarefree monomial ideal $I$ of $S$ such that $\depth S/I^{(k)}<\depth S/I^{k}$. Motivated by this example, we turn to the case of edge ideals. As a main theorem, we prove that its comparison holds for arbitrary graphs when second and third powers. Moreover, we study the module $I(G)^{(k)}/I(G)^{k}$ in order to obtain results on the depth comparison. In particular, we provide a complete characterization of $\dim I(G)^{(k)}/I(G)^{k}$ in graph-theoretic terms as an answer questions posed in \cite{hm2026}. First, for each $k\geq2$, we give an example that the inequality $\depth S/I_{\Delta}^{(k)}\geq \depth S/I_{\Delta}^{k}$ does not hold.

\begin{thm}\label{ex}
For any $k\geq2$, there exists a squarefree monomial ideal $I$ with $\height I=2$ such that 
$$\depth S/I^{(k)}<\depth S/I^{k}.$$
\end{thm}
\begin{proof}
Fix $k\geq2$. Set and let $G=W(K_{k+1})$ be the whisker graph of the complete graph $K_{k+1}$. Write $V(G)=\{x_1,\ldots,x_{k+1},y_1,\ldots,y_{k+1}\}$,
where $V(K_{k+1})=\{x_1,\ldots,x_{k+1}\}$ and $y_i$ is a leaf adjacent to $x_i$. Let $I=J(G)$ be the cover ideal of $G$. From \cite[Corollary 3.2]{fy2024}, we know $\depth S/I^{(k)}=(k+1)-1=k$. We now compute $\depth S/I^k$. Set $R=K[z_1,\ldots,z_{k+1}]$ and $Q=(z_1,\ldots,z_{k+1})^{2}$. Following \cite{bhhm2015}, let $\operatorname{Mon}(R\setminus Q)$ denote the set of monomials of $R$ which do not belong to $Q$. Then, we can easily verify that $\operatorname{Mon}(R\setminus Q)=\{1,z_{1},\ldots,z_{k+1}\}$. Since, for each $i$, the smallest positive integer $b_i$ such that $z_i^{b_i}\in Q$ is $b_i=2$, we have $R^{\rm pol}=K[z_{i,1},z_{i,2}:1\leq i\leq k+1]$, which we identify with $S$ by setting $z_{i,1}=x_{i}$ and $z_{i,2}=y_{i}$. Let $L(Q)$ be a monomial ideal generated by the monomials $z_{1,a_1+1}\cdots z_{k+1,a_{(k+1)}+1}$, where $z_1^{a_1}\cdots z_{k+1}^{a_{k+1}}\in\operatorname{Mon}(R\setminus Q)$. Now, since $\operatorname{Mon}(R\setminus Q)=\{1,z_{1},\ldots,z_{k+1}\}$, we see that $L(Q)=(x_1\cdots x_{k+1})+(y_i\prod_{j\neq i}x_j:1\leq i\leq k+1)$. On the other hand, since the set of minimal vertex covers of $G$ is 
$$\Min(G)=\{X_{[k+1]}\}\cup(\bigcup_{1\leq i\leq k+1}\{\{x_{1},\ldots,x_{k+1}, y_{i}\}\setminus\{x_{i}\}\}),$$ and hence, we see that $L(Q)=J(G)=I$.
From \cite[Corollary 11]{bhhm2015}, we obtain that 
\begin{align*}
\depth S/I^{k}&=\depth R^{\rm pol}/L(Q)^{k} \\
&=\sum_{1\leq i\leq k+1}b_{i}-\max\{\deg(\lcm(u_{1},\ldots, u_{k}))\,\,: u_{i}\in\operatorname{Mon}(R\setminus Q)\}-1 \\
&=2(k+1)-k-1 \\
&=k+1.
\end{align*}
Therefore, we obtain that 
$\depth S/I^{(k)}=k<k+1=\depth S/I^{k}$, as required. 
\end{proof}

From Theorem \ref{ex}, it is natural to ask when does the inequality hold for squarefree monomial ideals. We investigate this question for edge ideals. First, the following corollary is immediate from Theorem \ref{local coho}.

\begin{cor}
Let $G$ be a graph with $\ogirth(G)=2t-1$. If $1\leq k\leq t$, then we have $$\depth S/I(G)^{(k)}\geq\depth S/I(G)^{k}.$$ 
\end{cor}

It is known that Minh's conjecture hold for $k=2,3$, that is, $\reg I(G)^{(k)}=\reg I(G)^{k}$ for $k=2,3$ in \cite[Theorem 1.1]{mnptv2022}. Motivated by this result, we analogically investigate the depth comparison between small ordinary and symbolic powers. To this end, we prove an inclusion between the sets of associated radical ideals of second and third powers of edge ideals. Let us recall the notion of an associated radical ideal, which was introduced by Hochster in \cite{h1977}. Following the notation of Jafari and Sabzrou \cite{js2019}, for a monomial ideal $I\subset S$, we set 
$$\assrad(I)=\{\sqrt{I:u}\,\,: u\mbox{ is a monomial with }u\notin I\}.$$
An element of $\assrad(I)$ is called an {\it associated radical ideal of }$I$. 
To prove the inclusion, we prepare the following lemma:

\begin{lemma}\label{lemma for assrad}
Let $G$ be a graph on the vertex set $[n]$ and let ${\bf a}\in\mathbb{N}^{n}$ with ${\bf x}^{\bf a}\notin I(G)^{(3)}$. Suppose that $\sqrt{I(G)^{3}:{\bf x}^{\bf a}}\subsetneq\sqrt{I(G)^{(3)}:{\bf x}^{\bf a}}$. Then, for any $i\in\{i\,\,: x_{i}\in\mathcal{G}(\sqrt{I(G)^{(3)}:{\bf x}^{\bf a}}), x_{i}\notin \sqrt{I(G)^{3}:{\bf x}^{\bf a}}\}$, there exist $j_{i}\in V(G)$ and a triangle $T_{i}$ such that $\{i,j_{i}\}\in E(G)$, $j_{i}\notin T_{i}$, and $x_{j_{i}}{\bf x}_{T_{i}}$ divides ${\bf x}^{\bf a}$. 
\end{lemma}
\begin{proof}
Set $I=I(G)$, $P=\sqrt{I(G)^{(3)}:{\bf x}^{\bf a}}$, and $Q=\sqrt{I(G)^{3}:{\bf x}^{\bf a}}$. Also, set $U=\{i\,\,: x_{i}\in\mathcal{G}(P), x_{i}\notin Q\}$. From \cite[Lemma 4.1]{mnptv2022}, we see that $P=Q+(x_{i}\,\,: i\in U)$ and $i\notin\supp({\bf a})$ for each $i\in U$. Fix $i\in U$. From Case 3 in the proof of \cite[Lemma 4.1]{mnptv2022}, there exist $j\in V(G)$ and a triangle $T$ such that $\{i,j\}\in E(G)$, $i\notin T\cup\supp({\bf a})$, and $$x_{i}=\sqrt{\frac{x_{i}x_{j}{\bf x}_{T}}{\gcd(x_{i}x_{j}{\bf x}_{T}, {\bf x}^{\bf a})}}.$$ 
We now prove that $j\notin T$. Suppose not and write $T=\{j,p,q\}$. Then, since  $x_{j}{\bf x}_{T}$ divides ${\bf x}^{\bf a}$ and $(x_{i}x_{j})^{2}(x_{p}x_{q})=x_{i}^{2}x_{j}^{2}x_{p}x_{q}=x_{i}^{2}(x_{j}{\bf x}_{T})$, we see that $x_{i}^{2}{\bf x}^{\bf a}\in I^{3}$, that is, $x_{i}\in Q$, a contradiction. Therefore, by setting $j_{i}=j$ and $T_{i}=T$, the assertion follows. 
\end{proof}

\begin{thm}\label{assrad}
Let $G$ be a graph. Then we have 
$$\assrad(I(G)^{(k)})\subset\assrad(I(G)^{k})\mbox{ for }k=2,3.$$
\end{thm}
\begin{proof}
Set $I=I(G)$. For $k=2$, from \cite[Lemma 3.1]{mnptv2022}, we know $\sqrt{I^{(2)}:{\bf x}^{\bf a}}=\sqrt{I^{2}:{\bf x}^{\bf a}}$ for ${\bf x}^{\bf a}\notin I^{(2)}$, and hence the assertion follows. Suppose that $k=3$. Fix $P\in\assrad(I^{(3)})$. Then there exists ${\bf a}\in\mathbb{N}^{n}$ such that ${\bf x}^{\bf a}\notin I^{(3)}$ and $P=\sqrt{I^{(3)}:{\bf x}^{\bf a}}$. Set $Q=\sqrt{I^{3}:{\bf x}^{\bf a}}$ and $U=\{i\,\,: x_{i}\in\mathcal{G}(P), x_{i}\notin Q\}$. Moreover, from \cite[Lemma 4.1]{mnptv2022}, we have $i\notin\supp({\bf a})$ for each $i\in U$. Since $I^{3}\subset I^{(3)}$, we clearly have $Q\subset P$. If $P=Q$, then there is nothing to prove. Hence, we suppose that $Q\subsetneq P$. Then, from Lemma \ref{lemma for assrad}, for any $i\in U$, there exist $j_{i}\in V(G)$ and a triangle $T_{i}$ such that $\{i,j_{i}\}\in E(G)$, $j_{i}\notin T_{i}$, and $x_{j_{i}}{\bf x}_{T_{i}}$ divides ${\bf x}^{\bf a}$. Set $W=\{j_{i}\,\,; i\in U\}$ and define a vector ${\bf b}\in\mathbb{N}^{n}$ such that ${\bf x}^{\bf b}={\bf x}^{\bf a}\prod_{j\in W}x_{j}$. Then we claim that $P=\sqrt{I^{(3)}:{\bf x}^{\bf b}}$. Let $\Min(G)$ denotes the minimal vertex covers of $G$ and set $P_{C}=(x_{\ell}\,\,: \ell\in C)$ for every $C\in\Min(G)$. For any ${\bf c}\in\mathbb{N}^{n}$, by a straightforward computation, we have 
\[
\sqrt{P_{C}^{3}:{\bf x}^{\bf c}}=
\begin{cases}
P_{C}, & \text{if } \sum_{\ell\in C}c_{\ell}\leq2. \\
S, & \text{ otherwise.}
\end{cases}
\]
Hence, since $I^{(3)}=\bigcap_{C\in\Min(G)}P_{C}^{3}$, for any ${\bf c}\in\mathbb{N}^{n}$, we obtain that $$\sqrt{I^{(3)}:{\bf x}^{\bf c}}=\bigcap_{\substack{C\in\Min(G), \\\sum_{\ell\in C}c_{\ell}\leq2}}P_{C}.$$
Suppose that $\sum_{\ell\in C}a_{\ell}\leq2$ for some $C\in\Min(G)$. We prove that $C\cap W=\emptyset$. Suppose not. Then there exists $i\in U$ such that $j_{i}\in C$. Since $T_{i}$ is a triangle and $C$ is a vertex cover of $G$, we have $|C\cap T_{i}|\geq2$. Hence, we can take $p,q\in C\cap T_{i}$ with $p\neq q$. Since $x_{j_{i}}{\bf x}_{T_{i}}$ divides ${\bf x}^{\bf a}$, we have $a_{j_{i}}\geq1$, $a_{p}\geq1$, and $a_{q}\geq1$. Thus, we see that $\sum_{\ell\in C}a_{ell}\geq a_{j_{i}}+a_{p}+a_{q}\geq3$, a contradiction. Hence, by the definition of ${\bf b}$, we have $\sum_{\ell\in C}b_{\ell}=\sum_{\ell\in C}a_{\ell}\leq2$, which implies that $\sqrt{I^{(3)}:{\bf x}^{\bf b}}=\sqrt{I^{(3)}:{\bf x}^{\bf a}}=P$, in particular, we see that $\sqrt{I^{3}:{\bf x}^{\bf b}}\subset P$. Finally, we prove that $\sqrt{I^{3}:{\bf x}^{\bf b}}\supset P$. Since ${\bf x}^{\bf a}$ divides ${\bf x}^{\bf b}$, we have $Q\subset\sqrt{I^{3}:{\bf x}^{\bf b}}$. Fix $i\in U$ and write $T_{i}=\{p,q,r\}$. Now, since $x_{j_{i}}x_{p}x_{q}x_{r}$ divides ${\bf x}^{\bf a}$ and $j_{i}\in W$, $x_{j_{i}}^{2}x_{p}x_{q}$ divides ${\bf x}^{\bf b}$. On the other hand, since $\{i,j_{i}\},\{p,q\}\in E(G)$, we obtain $(x_{i}x_{j_{i}})^{2}(x_{p}x_{q})\in I^{3}$, and thus we see that $x_{i}^{2}{\bf x}^{\bf b}\in I^{3}$, that is, $x_{i}\in\sqrt{I^{3}:{\bf x}^{\bf b}}$. Therefore, we obtain that $P=Q+(x_{i}\,\,: i\in U)\subset\sqrt{I^{3}:{\bf x}^{\bf b}}$. Consequently,  we obtain that $\sqrt{I^{3}:{\bf x}^{\bf b}}=\sqrt{I^{(3)}:{\bf x}^{\bf a}}$. Since $P\neq S$, we have ${\bf x}^{\bf b}\notin I^{3}$, and thus, we get $P\in\assrad(I^{3})$, as required. 
\end{proof}

As a corollary of Theorem \ref{assrad} and the following theorem, we obtain the depth comparison for small ordinary and symbolic powers as an analogue of the results for the regularity given in \cite[Theorem 1.1]{mnptv2022}. 

\begin{thm}[\cite{h1977}]\label{Hocster}
For a monomial ideal $I$ of $S$, we have 
$$\depth S/I=\min\{\depth S/J\,\,: J\in\assrad(I)\}.$$
\end{thm}

\begin{cor}\label{small power}
For any graph $G$, we have 
$$\depth S/I(G)^{(k)}\geq\depth S/I(G)^{k}\mbox{ for }k=2,3.$$
\end{cor}

Moreover, for the second and fourth powers, we give a sufficient condition for the depth inequality to hold for all powers.

\begin{prop}\label{sufficient}
Let $G$ be a connected graph. If $\depth S/I(G)^{2}\leq1$ or $\depth S/I(G)^{4}=0$, then we have  
$$\depth S/I(G)^{(k)}\geq\depth S/I(G)^{k}\mbox{ for all }k\geq1.$$
\end{prop}
\begin{proof}
Set $I=I(G)$. We may assume that $G$ is non-bipartite from Lemma \ref{char of symb=ord}. Notice that $\depth S/I^{(k)}\geq1$ for all $k\geq1$. Suppose that $\depth S/I^{4}=0$. Then, from \cite[Theorem 2.15]{mmv2012}, we see that $\depth S/I^{k}=0$ for all $k\geq4$, and hence, from Corollary \ref{small power}, we obtain the desired inequality for all $k\geq1$. Therefore, we assume that $\depth S/I^{2}\leq1$. If $\depth S/I^{2}=0$, then, again from \cite[Theorem 2.15]{mmv2012}, the assertion follows. Hence, we may assume that $\depth S/I^{2}=1$. In this case, from \cite[Theorem 4.8]{tt2014}, we have either $\diam(G^{c})\geq3$ or there exists a triangle $T$ such that $|W|\leq1$ or $G^{c}|_{W}$ is disconnected, where we set $W=V(G)\setminus N[T]$. 
We distinguish the following cases: 

{\bf Case 1.} Suppose that $\diam(G^{c})\geq3$. In this case, from \cite[Proposition 4.5]{hlt2023}, it is known that $\depth S/I^{3}=0$. Hence, from Corollary \ref{small power} and \cite[Theorem 2.15]{mmv2012}, the assertion follows. 

{\bf Case 2.} Suppose that there exists a triangle $T$ such that $|W|\leq1$. Since $\depth S/I^{2}\neq0$ and \cite[Theorem 2.8]{tt2014}, we see that $|W|=1$, and hence, we set $W=\{u\}$. Set $H=G-N[u]$. Notice that $T\subset V(H)$. Also, for $v\in V(H)\setminus T$, since $v\neq u$ and $W=\{u\}$, we get $v\in N[T]$. Thus, we get $N_{H}[T]=V(H)$. Therefore, we get $\depth K[V(H)]/I(H)^{2}=0$ from \cite[Theorem 2.8]{tt2014}. From \cite[Proposition 1.6, Lemma 1.7, Proposition 3.1]{hlt2023}, we see that $H_{\mathfrak{m}}^{1}(S/I^{k})\neq0$ for all $k\geq2$, which means that $\depth S/I^{k}\leq1$ for all $k\geq2$. 

{\bf Case 3.} Suppose that there exists a triangle $T$ such that $G^{c}|_{W}$ is disconnected. From \cite[Theorem 4.4]{hlt2023}, Corollary \ref{small power}, and \cite[Theorem 2.15]{mmv2012}, it is enough to prove that $\depth S/I^{4}\leq1$. We may write $T=\{x_{1},x_{2},x_{3}\}$ and let ${\bf a}={\bf e}_{1}+{\bf e}_{2}+{\bf e}_{3}$. Now, from \cite[Lemma 4.7]{tt2014}, we know $\Delta_{\bf a}(I^{2})=\Delta(G|_{W})$. Since the 1-skeleton of $\Delta(G|_{W})$ is $G^{c}|_{W}$ and the assumption that $G^{c}|_{W}$ is disconnected, we obtain that $\Delta_{\bf a}(I^{2})$ is disconnected. Suppose that $C_{1}$ and $C_{2}$ are disconnected components of $\Delta_{\bf a}(I^{2})$ and take vertices $u_{1}\in C_{1}$ and $u_{2}\in C_{2}$. Since $\{u_{1},u_{2}\}\notin E(G^{c})$, we have $\{u_{1},u_{2}\}\in E(G)$, and hence, from  the proof of \cite[Proposition 3.2]{hlt2023}, there exists a vector ${\bf b}\in\mathbb{N}^{n}$ such that $\Delta_{\bf b}(I^{3})\subset\Delta_{\bf a}(I^{2})$ and $u_{1},u_{2}\in\Delta_{\bf b}(I^{3})$. Hence, $u_{1}$ and $u_{2}$ still belongs to disconnected components of $\Delta_{\bf b}(I^{3})$. Since $\{u_{1},u_{2}\}\in E(G)$, from \cite[Proposition 3.2]{hlt2023}, there exists a vector ${\bf c}\in\mathbb{N}^{n}$ such that $\Delta_{\bf c}(I^{4})$ is disconnected, that is, $H_{\mathfrak{m}}^{1}(S/I^{4})\neq0$, which means $\depth S/I^{4}\leq1$, as required. 
\end{proof}

We now turn to the module $I(G)^{(k)}/I(G)^{k}$ , which will be used in our depth comparison. First, we determine when the module $I(G)^{(k)}/I(G)^{k}$ remains after localization at a monomial prime ideal. For a graph $G$ and $k\geq1$, we set 
$$\mathcal{C}_{k}(G)=\{C\,\,: C\mbox{ is an induced odd cycle of }G\mbox{ with }|C|\leq 2k-1\}.$$

\begin{lemma}\label{criteria of neq0}
Let $G$ be a graph and $A$ be an independent set of $G$, Then, for each $k\geq1$, $(I(G)^{(k)}/I(G)^{k})_{P_{A}}\neq0$ if and only if $G-N[A]$ has an induced odd cycle of length at most $2k-1$, where $P_{A}=(x_{i}\,\,: i\notin A)$. 
\end{lemma}
\begin{proof}
Set $I=I(G)$ and fix $k\geq1$. Let $H=G-N[A]$, $Q=(x_{i}\,\,: i\in N(A))$. Then we know $IS_{P_{A}}=(Q+I(H))S_{P_{A}}$, $I^{k}S_{P_{A}}=(IS_{P_{A}})^{k}$, and $I^{(k)}S_{P_{A}}=(IS_{P_{A}})^{(k)}$. Hence, $(I^{(k)}/I^{k})_{P_{A}}\neq0$ if and only if $(Q+I(H))^{(k)}S_{P_{A}}\neq(Q+I(H))^{k}S_{P_{A}}$. Now, we have $(Q+I(H))^{k}=\sum_{0\leq j\leq k}Q^{k-j}I(H)^{j}$. Moreover, from \cite[Theorem 3.4]{hntt2020}, we know that $(Q+I(H))^{(k)}=\sum_{0\leq j\leq k}Q^{k-j}I(H)^{(j)}$. Now, we assume that $H$ has no induced odd cycle of length at most $2k-1$. Then from Lemma \ref{char of symb=ord}, we have $I(H)^{(j)}=I(H)^{j}$ for all $0\leq j\leq k$, and hence, we obtain that 
$$(Q+I(H))^{(k)}=\sum_{0\leq j\leq k}Q^{k-j}I(H)^{(j)}=\sum_{0\leq j\leq k}Q^{k-j}I(H)^{j}=(Q+I(H))^{k},$$which implies that $(I^{(k)}/I^{k})_{P_{A}}=0$. Conversely, we assume that $H$ has an induced odd cycle of length at most $2k-1$. Again from Lemma \ref{char of symb=ord}, we have $I(H)^{(k)}\neq I(H)^{k}$. Hence, there exists a monomial $u$ such that $u\in I(H)^{(k)}\setminus I(H)^{k}$. Now, since $u\in I(H)^{(k)}$ and $u\notin Q^{k-j}I(H)^{j}$ for each $j<k$, we see that $u\notin\sum_{0\leq j\leq k}Q^{k-j}I(H)^{j}=(Q+I(H))^{k}$. Suppose that $u/1\in(Q+I(H))^{k}S_{P_{A}}$. Then there exists $f\notin P_{A}$ such that $fu\in(Q+I(H))^{k}S$. Since $f\notin P_{A}$, there exists a monomial $v\in K[x_{i}\,\,: i\in A]$ such that $uv\in(Q+I(H))^{k}S$, which implies that $u\in(Q+I(H))^{k}$, a contradiction. Hence, $u/1\in(Q+I(H))^{(k)}S_{P_{A}}\setminus(Q+I(H))^{k}S_{P_{A}}$, and hence, we get $(I^{(k)}/I^{k})_{P_{A}}\neq0$, as required. 
\end{proof}

Using Lemma \ref{criteria of neq0}, we obtain an explicit description of the support of $I(G)^{(k)}/I(G)^{k}$ as follows. 

\begin{thm}\label{locus}
For any graph $G$ and $k\geq1$, we have 
$$\sqrt{I(G)^{k}:I(G)^{(k)}}=\sqrt{\Ann_{S}I(G)^{(k)}/I(G)^{k}}=\bigcap_{C\in\mathcal{C}_{k}(G)}(I(G)+(x_{i}\,\,: i\in N_{G}[C]))$$
\end{thm}
\begin{proof}
Fix $k\geq1$, and set $I=I(G)$, $J_{C}=(x_{i}\,\,: i\in N_{G}[C])$ for each $C\in\mathcal{C}_{k}(G)$. For $A\subset V(G)$, we set $P_{A}=(x_{i}\,,: i\notin A)$. Fix $A\subset V(G)$. We claim that $(I^{(k)}/I^{k})_{P_{A}}\neq0$ if and only if $\bigcap_{C\in\mathcal{C}_{k}}(I+J_{C})\subset P_{A}$. First, suppose that $A$ is not independent. Then there exists $i,j\in A$ such that $\{i,j\}\in E(G)$. Notice that $x_{i}x_{j}\in I(G)$ and $x_{i}x_{j}\notin P_{A}$, and hence $IS_{P_{A}}=S_{P_{A}}$. Thus, we obtain that $I^{k}S_{P_{A}}=S_{P_{A}}=I^{(k)}S_{P_{A}}$, which implies that $(I^{(k)}/I^{k})_{P_{A}}=0$. On the other hand, since $A$ is not independent, we have $I\not\subset P_{A}$, and hence, we get $\bigcap_{C\in\mathcal{C}_{k}(G)}(I+J_{C})\not\subset P_{A}$. Next, suppose that $A$ is independent. From Lemma \ref{criteria of neq0}, we know $(I^{(k)}/I^{k})_{P_{A}}\neq0$ if and only if there exists $C\in\mathcal{C}_{k}(G)$ such that $V(C)\subset V(G)\setminus N[A]$. Moreover, we can easily verify that $V(C)\subset V(G)\setminus N[A]$ if and only if $A\cap N[C]=\emptyset$. Since $A\cap N[C]=\emptyset$ is equivalent to $J_{C}\subset P_{A}$, and hence, we see that $A\cap N[C]=\emptyset$ if and only if $J_{C}\subset P_{A}$. Thus, $(I^{(k)}/I^{k})_{P_{A}}\neq0$ if and only if there exists $C\in\mathcal{C}_{k}(G)$ such that $J_{C}\subset P_{A}$. Moreover, the latter one is equivalent to the condition $\bigcap_{C\in\mathcal{C}_{k}(G)}(I+J_{C})\subset P_{A}$. Therefore, for a monomial prime ideal $P_{A}$, we see that $\sqrt{\Ann_{S}I^{(k)}/I^{k}}\subset P_{A}$ if and only if $\bigcap_{C\in\mathcal{C}_{k}(G)}(I+J_{C})\subset P_{A}$, which implies that $$\sqrt{\Ann_{S}I^{(k)}/I^{k}}=\bigcap_{C\in\mathcal{C}_{k}}(I+J_{C}),$$which completes the proof.  
\end{proof}

Theorem \ref{locus} gives a precise combinatorial formula for $\dim I(G)^{(k)}/I(G)^{k}$ for each $k\geq2$. As a consequence, we obtain the following corollary, which recovers Corollary 7.8 and Theorems 7.10 and 7.11 as immediate consequence, and yields a complete answer to Question 7.9 and Problem 7.13 in \cite{hm2026}. 

\begin{cor}\label{char of dim}
Let $G$ be a graph and $k\geq1$. Then we have 
$$\dim I(G)^{(k)}/I(G)^{k}=\max\{\alpha(G-N_{G}[C])\,\,: C\in\mathcal{C}_{k}(G)\},$$
where we use the convention that $\max\emptyset=\dim 0=-\infty$. 
\end{cor}
\begin{proof}
Fix $k\geq1$ and set $I=I(G)$. We may assume that $\mathcal{C}_{k}(G)\neq\emptyset$. From Theorem \ref{locus}, we see that 
\begin{align*}
\dim I^{(k)}/I^{k}&=\dim S/\Ann_{S}I^{(k)}/I^{k}\\
&=\dim S/\sqrt{\Ann_{S}I^{(k)}/I^{k}} \\
&=\dim S/(\cap_{C\in\mathcal{C}_{k}(G)}(I+(x_{i}\,\,: i\in N[C]))) \\
&=\max_{C\in\mathcal{C}_{k}(G)}\{\dim S/(I+(x_{i}\,\,: i\in N[C]))\} \\
&=\max\{\alpha(G-N_{G}[C])\,\,: C\in\mathcal{C}_{k}(G)\},
\end{align*}
which completes the proof. 
\end{proof}

We finally apply Corollary \ref{char of dim} to the depth comparison. 

\begin{cor}\label{depth comparison free}
Let $G$ be a  graph and $k\geq2$. Suppose that $G$ has no induced subgraph isomorphic to  $(C_{2r+1}\sqcup2K_{1})$ for all $1\leq r\leq k-1$. Then we have $\depth S/I(G)^{(k)}\geq\depth S/I(G)^{k}$.  
\end{cor}
\begin{proof}
Set $I=I(G)$ and fix $k\geq2$. We may assume that $I^{(k)}\neq I^{k}$. First, we claim that $\alpha(G-N[C])\leq1$ for all $C\in\mathcal{C}_{k}(G)$. Fix $C\in\mathcal{C}_{k}(G)$. Then we may assume that $|C|=2r+1$ for some $1\leq r\leq k-1$. Suppose that $\alpha(G-N[C])\geq2$. Then there exist vertices $u,v\in V(G)\setminus N[C]$ such that $u\neq v$ and $\{u,v\}\notin E(G)$. Then the induced subgraph $G|_{V(C)\cup\{u,v\}}$ is isomorphic to $C_{2r+1}\sqcup2K_{1}$, a contradiction. Hence, from Corollary \ref{char of dim}, we obtain $\dim I^{(k)}/I^{k}\leq1$. From a short exact sequence $0\rightarrow I^{(k)}/I^{k}\rightarrow S/I^{k}\rightarrow S/I^{(k)}\rightarrow0$, we know 
$$1\geq\depth I^{(k)}/I^{k}\geq\min\{\depth S/I^{k}, \depth S/I^{(k)}+1\}.$$
Since $\depth S/I^{(k)}\geq1$, we obtain $\depth S/I^{k}\leq1$, as required, 
\end{proof}

As a direct application of Corollary \ref{depth comparison free}, we obtain the depth comparison for complete multipartite graphs. 

\begin{cor}
Let $G$ be a complete multi-partite graph. Then we have 
$$\depth S/I(G)^{(k)}\geq\depth S/I(G)^{k}\mbox{ for all }k\geq1.$$
\end{cor}


\section*{Acknowledgement}
The author is deeply grateful to his supervisor, Naoki Terai, for his valuable comments and helpful advice. The research was partially supported by ohmoto-ikueikai and JST SPRING Japan Grant Number JPMJSP2126. 



\end{document}